\documentclass[sort&compress]{elsarticle}   

\journal{Journal of Computational Physics}

\usepackage{graphicx,pgfplots}
\usepackage{bbm}
\usepackage{amsmath,amsfonts,amssymb}
\usepackage{lscape}
\usepackage{subcaption}
\usepackage{hyperref}
\usepackage{cleveref}
\usepackage{mathtools}
\usepackage{arydshln} % Package for dashed lines
\usepackage{lineno}
\usepackage{xparse}
\usepackage{mleftright}
\usepackage{csl}
\usepackage[authormarkup=none]{changes}
\definechangesauthor[name=Reviewer1, color=red]{1}
\definechangesauthor[name=Reviewer2, color=blue]{2}

\graphicspath{{./Figures/}}
\newcommand{\datadir}{./Data}

\crefname{equation}{}{}
\Crefname{equation}{Equation}{Equations}
\crefname{appendix}{}{}

\DeclareMathOperator{\diag}{diag} 

\newtheorem{theorem}{Theorem}
\newdefinition{definition}{Definition}
\newdefinition{remark}{Remark}
\newdefinition{example}{Example}
\newproof{proof}{Proof}

\newenvironment{butchertableau}[2][1.25]{\def\arraystretch{#1}\array{#2}}{\endarray}
\newenvironment{smatrix}[1][1.15]{\def\arraystretch{#1}\begin{bmatrix}}{\end{bmatrix}}

\NewDocumentCommand{\eye}{g g}{\mathbf{I}\IfValueT{#1}{_{{#1} \times \IfValueTF{#2}{#2}{#1}}}}
\NewDocumentCommand{\zero}{g g}{\IfValueTF{#2}{\mathbf{0}}{0}\IfValueT{#1}{_{{#1} \IfValueT{#2}{\times {#2}}}}}
\NewDocumentCommand{\one}{g g}{\mathbbm{1}\IfValueT{#1}{_{{#1} \IfValueT{#2}{\times {#2}}}}}
\NewDocumentCommand{\unitvec}{m}{\mathbf{e}_{#1}}
\NewDocumentCommand{\R}{g g g}{\mathbb{R}^{\IfValueT{#1}{#1} \IfValueT{#2}{\times {#2}} \IfValueT{#3}{\times {#3}}}}
\NewDocumentCommand{\Cplx}{g g g}{\mathbb{C}^{\IfValueT{#1}{#1} \IfValueT{#2}{\times {#2}} \IfValueT{#3}{\times {#3}}}}
\NewDocumentCommand{\tril}{g g}{\mathbf{T}\IfValueT{#1}{_{{#1} \times \IfValueTF{#2}{#2}{#1}}}}

\NewDocumentCommand{\component}{m m m}{^{ \left\{ #1 \right\} #2 \IfBooleanT{#3}{\tp}}}
\NewDocumentCommand{\comp}{m O{} s}{\component{#1}{#2}{#3}}

\RenewDocumentCommand{\L}{O{} s}{\component{\textsc{l}}{#1}{#2}}
\NewDocumentCommand{\D}{O{} s}{\component{\textsc{d}}{#1}{#2}}
\NewDocumentCommand{\U}{O{} s}{\component{\textsc{u}}{#1}{#2}}

\NewDocumentCommand{\E}{O{} s}{\component{\textsc{e}}{#1}{#2}}
\NewDocumentCommand{\I}{O{} s}{\component{\textsc{i}}{#1}{#2}}

\newcommand{\nvar}{\mathrm{d}} 
\newcommand{\nparts}{\mathrm{N}}
\newcommand{\tp}{T}

\newcommand{\Dt}{h}
\newcommand{\A}{\boldsymbol{A}}
\renewcommand{\b}{\boldsymbol{b}}
\renewcommand{\c}{\boldsymbol{c}}
\newcommand{\Ahat}{\boldsymbol{\widehat{A}}}

\newcommand{\Atilde}{\boldsymbol{\widetilde{A}}}
\newcommand{\btilde}{\boldsymbol{\widetilde{b}}}
\newcommand{\ctilde}{\boldsymbol{\widetilde{c}}}

\pgfplotsset{compat=1.15}

\pgfplotscreateplotcyclelist{mycolor}{%
	blue,every mark/.append style={fill=blue!80!black},mark=*\\%
	red,every mark/.append style={fill=red!80!black},mark=square*\\%
	teal,every mark/.append style={fill=teal!80!black},mark=triangle\\%
	black,mark=star\\%
	blue,every mark/.append style={fill=blue!80!black},mark=diamond*\\%
	red,densely dashed,every mark/.append style={solid,fill=red!80!black},mark=*\\%
	brown!60!black,densely dashed,every mark/.append style={
		solid,fill=brown!80!black},mark=square*\\%
	black,densely dashed,every mark/.append style={solid,fill=gray},mark=otimes*\\%
	blue,densely dashed,mark=star,every mark/.append style=solid\\%
	red,densely dashed,every mark/.append style={solid,fill=red!80!black},mark=diamond*\\%
}

\pgfplotsset{every axis/.append style={
	width=\linewidth,
	x label style={font=\scriptsize},
	y label style={font=\scriptsize},
	tick label style={font=\scriptsize},
	legend columns=1,
	legend style={at={(0.5,1.05)}, anchor=south, font=\scriptsize},
	cycle list name=mycolor,
	yminorticks=false,
	mark size=1.4pt,
	max space between ticks=25pt,
	try min ticks=3
}}

\renewcommand{\comment}[2][]{}%
\renewcommand{\added}[2][]{#2}%
\renewcommand{\deleted}[2][]{}

\newif\ifreport
\reportfalse

\begin{document}
    % Tech report title page
    \csltitle{A unified formulation of splitting-based  implicit time integration schemes} 
    \cslauthor{S. Gonz\'alez-Pinto, D. Hern\'andez-Abreu, M. S. P\'erez-Rodr\'iguez, 
        A. Sarshar, S. Roberts, A. Sandu}
    \cslyear{21}
    \cslreportnumber{3}
    \csltitlepage
     %  Journal 
	\begin{frontmatter}

		\title{A unified formulation of splitting-based \\ implicit time integration schemes}
		
		\author[lg]{Severiano Gonz\'alez-Pinto\fnref{fnSeveriano}}\ead{spinto@ull.edu.es}
		\author[lg]{Domingo Hern\'andez-Abreu\fnref{fnSeveriano}}\ead{dhabreu@ull.edu.es}
		\author[lg]{Maria S. P\'erez-Rodr\'iguez\fnref{fnSeveriano}}\ead{sperezr@ull.edu.es}
		\author[vt]{Arash Sarshar\fnref{fnSandu}}\ead{sarshar@vt.edu}
		\author[vt]{Steven Roberts\corref{corA}\fnref{fnSandu}}\ead{steven94@vt.edu} 
		\author[vt]{Adrian Sandu\fnref{fnSandu}}\ead{sandu@cs.vt.edu}
		
        \fntext[fnSeveriano]{The work of these authors has been partially supported by University of La Laguna, Canary Islands.}
        
		\cortext[corA]{Corresponding author}
		\address[lg]{Departamento de An\'alisis Matem\'atico, Universidad de La Laguna, 38200-La Laguna, Spain}
		\address[vt]{Computational Science Laboratory, Department of Computer Science, Virginia Tech, Blacksburg, VA 24060}
		\fntext[fnSandu]{The work of Sandu, Sarshar, and Roberts has been supported in part by NSF through awards NSF ACI--1709727, NSF CDS\&E--MSS 1953113, DOE ASCR DE--SC0021313, and by the Computational Science Laboratory at Virginia Tech.}
        
        \begin{abstract}
            Splitting-based time integration approaches such as fractional step, alternating direction implicit, operator splitting, and locally one dimensional methods partition the system of interest into components, and solve individual components  implicitly in a cost-effective way. 
            This work proposes a unified formulation of splitting time integration schemes in the framework  of general-structure additive Runge--Kutta (GARK) methods. Specifically, we develop implicit-implicit (IMIM) GARK schemes, provide the order conditions for this class, and explain their application to partitioned systems of ordinary differential equations.  We show that classical splitting methods belong to the IMIM GARK family, and therefore can be studied in this unified framework.  New IMIM-GARK splitting methods are developed and tested using parabolic systems. 
        \end{abstract}
		
		
	\end{frontmatter}

	%!TEX root = ../main.tex
%%%%%%%%%%%%%%%%%%%%%%%%%%
\section{Introduction}
%%%%%%%%%%%%%%%%%%%%%%%%%%
Many applications require the solution of partitioned initial value problems 
\begin{equation} 
	\label{eqn:additive-ode}
	y'= f(t,y) = \sum_{m=1}^\nparts f\comp{m} (t,y), \qquad  y(t_0)=y_0 \in \R{\nvar},
\end{equation}
where the right-hand side function $f: \R{\nvar} \rightarrow \R{\nvar}$ is additively  split into $\nparts$ parts. The partitioning may  be informed by spatial dimensions, physical processes, stiffness, linearity,  time-scales, or computational costs. Additive partitioning also includes the special case of component partitioning where the state vector $y$ is split into disjoint sets \cite{Sandu_2015_GARK}. 

The earliest approaches to efficiently tackle partitioned problems \cref{eqn:additive-ode} involve solving the individual component systems one at a time, in an alternating fashion. This idea led to several closely connected families of schemes including fractional step methods, alternating direction implicit (ADI) methods, operator splitting methods, and locally one dimensional (LOD) methods. An early survey of fractional step methods is given in Yanenko's monograph \cite{Yanenko_1971_book}. ADI methods for the heat equation were proposed by Douglas \cite{Douglas_1955_ADI}, Douglas and Rachford \cite{Douglas_1956_ADI}, and Peaceman and Rachford \cite{Peaceman_1955_ADI}; the idea is to solve multidimensional partial differential equation via a sequence of one-dimensional problems, each aligned to a different spatial dimension. Recent theoretical results on stabilibity and convergence for this kind of methods are obtained in \cite{Gonzalez_2020_AMF,NR4, NR5}. \added[id=1]{These splittings are further explored in \cite{Tan2007LOD,Tan2008ADI,Tan2020Review} for efficient time-stepping of Maxwell's equations with special interest in reducing right-hand side function evaluations.} Bujanda and Jorge \cite{Bujanda_2003_FSRK,Bujanda_2001_FSRK} propose fractional step Runge--Kutta methods (FSRK) for solving multidimensional  parabolic PDEs, by means of linearly implicit time integration processes. They extend these methods to semi-linear parabolic problems in an IMEX/ARK fashion \cite{Bujanda_2004_FSRK,Bujanda_2007_FSRK} where a FSRK implicit method used to solve the linear non-homogeneous terms is paired with an explicit Runge--Kutta method for solving nonlinear terms.  An extension of ADI strategies for parabolic problems to general linear methods is discussed in \cite{Sandu_2019_GLM-ADI}. The traditional first order operator splitting method \cite{Yanenko_1971_book} was extended to a symmetric, second order splitting approach by Strang \cite{Strang_1968_splitting}, and to higher order splitting approaches by Yoshida \cite{Yoshida_1990_splitting}. 

The development of Runge--Kutta methods that are tailored to the partitioned system \cref{eqn:additive-ode} started with the early work of Rice \cite{Rice_1960_split-RK}. Hairer \cite{Hairer_1981_Pseries} developed the concept of P-trees and laid the foundation for the modern order conditions theory for partitioned Runge--Kutta methods. The investigation of practical partitioned Runge--Kutta methods with implicit and explicit components was revitalized by the work of Ascher, Ruuth, and Spiteri \cite{Ascher_1997_IMEX_RK}. \added[id=1]{The Additive Runge--Kutta framework \cite{Rentrop_1985,Cooper_1983_ARK} has served as the foundation for numerous later works on implicit-explicit Runge--Kutta methods \cite{Kennedy_2003_ARK,Boscarino2007IMEX-DAE,BOSCARINO2009IMEX}}. Some partitioning strategies have been discussed by Weiner et al. \cite{Weiner_1993_partitioning}. \added[id=1]{IMEX extrapolation methods \cite{Sandu_2014_IMEX-RK,Sandu_2010_extrapolatedIMEX,Sandu_2014_IMEX_GLM_Extrap,SCHNEIDER2018-IMEX-PEER-EXT} and IMEX general linear methods \cite{Sandu_2015_IMEX-TSRK,Sandu_2020_IMEX-GLM-Theory,Bras2017-IMEX-GLM,SOLEIMANI2017-IMEX-PEER,SOLEIMANI2018-IMEX-PEER,Sandu_2014_IMEX-GLM,SCHNEIDER2021-PEER-IMEX} are some  of the latest examples in this area.} Sandu and G{\"u}nther have proposed a generalized-structure for  additively partitioned Runge–Kutta (GARK) methods in \cite{Sandu_2015_GARK}.  This framework provides a unifying theory for the creation and analysis of partitioned Runge--Kutta methods in various applications such as multirate \cite{Sandu_2016_GARK-MR,Sandu_2019_MR-GARK_High-Order}, multirate infinitesimal \cite{Sandu_2019_MRI-GARK,Sandu_2020_MRI-GARK_Coupled}, differential-algebraic \cite{Sandu_2014_SCEE}, and linearly implicit \cite{Sandu_2021_GARK-ROS} integration methods; research implementations are available in \cite{Sandu_2014_FATODE,Sandu_2017_MATLODE}. An alternative approach to efficiently solving \cref{eqn:additive-ode} is to employ splitting at a linear algebra level, rather than splitting the system or its discretization; this led to approximate matrix factorization (AMF) family of schemes \cite{Beam_1976_FD,Gonzalez_2018_AMF,Gonzalez_2020_AMF,NR1,NR2,NR3,Sandu_2015_AMF-RK}.

This paper proposes a unified formulation of implicit integration schemes based on splitting (fractional step, alternating direction implicit, operator splitting, and locally one dimensional methods) in the framework of general-structure additive Runge--Kutta (GARK) methods. The framework allows for the construction of new schemes of higher classical order suited for serial and parallel partitioned integration.

The remainder of the paper is organized as follows. In \Cref{sec:imim}, practical structures for implicit integration are introduced, and \added[id=2]{\cref{sec:Linear-stability} provides insights into the stability of these methods.}  \Cref{sec:classical-adi-methods,sec:operator-splitting-methods,sec:FSRK} are dedicated to deriving GARK formulations of classical ADI schemes, operator splitting schemes and fractional step Runge--Kutta methods.  The order condition theory for implicit-implicit GARK methods is discussed in \cref{sec:new_methods}, and two new methods are derived.  Finally we present some numerical tests for the new methods in \Cref{sec:numerical_experiments}. 
	%!TEX root = ../main.tex
%%%%%%%%%%%%%%%%%%%%%%%%%%%%%%
\section{Implicit-implicit GARK schemes}
\label{sec:imim}
%%%%%%%%%%%%%%%%%%%%%%%%%%%%%%

We start by considering an $\nparts$-way partitioned GARK method \cite{Sandu_2015_GARK} to solve \cref{eqn:additive-ode}: 
\begin{subequations}
    \label{eqn:GARK}
    \begin{align}
	    \label{eqn:GARK-stage}
	    Y_i\comp{q} &= y_{n} + \Dt \sum_{m=1}^\nparts \sum_{j=1}^{s\comp{m}} a\comp{q,m}_{i,j} \, f\comp{m}_j, \\
	    & \qquad q =1,\dots,\nparts, \quad i=1,\dots,s\comp{q}, \nonumber \\
	    \label{eqn:GARK-solution}
	    y_{n+1} &= y_n + \Dt \sum_{q=1}^\nparts \sum_{i=1}^{s\comp{m}} b\comp{q}_{i} \, f\comp{q}_i,
    \end{align}
\end{subequations}
where we denote
\begin{equation*}
	f\comp{m}_j \coloneqq f\comp{m}\mleft(t_n + c\comp{m}_j \, \Dt, Y\comp{m}_j\mright),\quad j=1,\ldots,s\comp{m}.
\end{equation*}
The GARK method \cref{eqn:GARK} is represented by the Butcher tableau
\begin{equation}
	\label{eqn:consistent-Butcher-tableau}
	\begin{butchertableau}{c}
		 \A \\ \hline
		 \b^\tp
	\end{butchertableau}
	~\coloneqq~
	\raisebox{24pt}{$
	\begin{butchertableau}{cccc}
		 \A\comp{1,1} & \A\comp{1,2} & \ldots & \A\comp{1,\nparts} \\
		 \A\comp{2,1} &\A\comp{2,2} & \ldots & \A\comp{2,\nparts} \\
		  \vdots & \vdots & & \vdots \\
		  \A\comp{\nparts,1} & \A\comp{\nparts,2} & \ldots & \A\comp{\nparts,\nparts} \\ \hline
		 \b\comp{1}* & \b\comp{2}* & \ldots &\b\comp{\nparts}*
	\end{butchertableau}
	$},
\end{equation}
\added[id=2]{where $\A\comp{q,m} \in \R{s\comp{q}}{s\comp{m}}$ and $\b\comp{q}, \c\comp{q} \in \R{s\comp{q}}$ for $q, m = 1, \dots, \nparts$.  Thus, $\A \in \R{\mathbf{s}}{\mathbf{s}}$ and $\b \in \R{\mathbf{s}}$, where $\mathbf{s} = \sum_{m=1}^{\nparts} s\comp{m}$.  The diagonal blocks of the tableau correspond to the ``base'' methods which are traditional Runge--Kutta schemes with coefficients $\left( \A\comp{q,q}, \b\comp{q}, \c\comp{q} \right)$.}

The GARK order conditions up to order four are given in \cref{eqn:GARK-order-conditions} following \cite{Sandu_2015_GARK}.  The GARK scheme \cref{eqn:consistent-Butcher-tableau} is internally consistent \cite{Sandu_2015_GARK} if
\comment[id=2]{Remove ambiguous $c\comp{q,m}_i$ and more precisely define int. consistency}
\begin{equation}
	\label{eqn:internal-consistency}
	%c\comp{q,m}_i \coloneqq \sum_{j=1}^{s\comp{m}} a\comp{q,m}_{i,j},
	\A\comp{q,1} \one{s\comp{1}} = \dots = \A\comp{q,\nparts} \one{s\comp{\nparts}} = \c\comp{q}.
\end{equation}
Internal consistency greatly simplifies the number of coupling order conditions \cref{eqn:GARK-order-conditions} \cite{Sandu_2015_GARK}.

The GARK scheme \cref{eqn:consistent-Butcher-tableau} is called stiffly accurate \cite{Sandu_2015_GARK} if its coefficients satisfy
\begin{equation}
	\label{eqn:gark-stiff-accuracy}
	%\btilde^\tp = \unitvec{\mathbf{s}}^\tp \, \Atilde
	%\quad \Leftrightarrow \quad
	\b^\tp = \unitvec{\mathbf{s}}^\tp \, \A
    \quad \Leftrightarrow \quad
    \b\comp{m}* = \unitvec{s\comp{\nparts}}^\tp \, \A\comp{\nparts,m}, \text{ for } m=1,\dots,\nparts,
\end{equation}
where $\unitvec{i} \in \R{\nparts}$ is the $i$th column of identity matrix $\eye{\nparts}$.  In cases where the dimension is ambiguous, we write this as $\unitvec{i,N}$.
Stiffly accurate GARK methods have favorable stability properties \cite{Sandu_2015_GARK} and simpler order conditions for differential algebraic equations \cite{Tanner2018}.  Note in some cases, a GARK method may not satisfy \cref{eqn:gark-stiff-accuracy}, but a permutation of stages can yield an equivalent formulation that does.   More formally, denote a permutation matrix by $\mathcal{P} \in \R{\mathbf{s}}{\mathbf{s}}$. Reordering the stage numbers accordingly leads to permuted coefficients and a rearranged Butcher tableau \cref{eqn:consistent-Butcher-tableau}
\begin{equation}
	\label{eqn:permuted-Butcher-tableau}
	\begin{butchertableau}[1.6]{c}
		 \Atilde \\
		\hline
		 \btilde^\tp
	\end{butchertableau}
	~~\coloneqq~~
	\begin{butchertableau}{c}
		\mathcal{P} \, \A \, \mathcal{P}^\tp \\
		\hline
		\b^\tp \,\mathcal{P}^\tp
	\end{butchertableau}.
\end{equation}
We are now ready to define implicit-implicit GARK schemes.
\begin{definition}[Implicit-implicit GARK schemes]
	\label{def:imim-gark}
	A GARK scheme \cref{eqn:consistent-Butcher-tableau} is implicit-implicit (IMIM-GARK) if there exists a permutation  $\mathcal{P} \in \R{\mathbf{s}}{\mathbf{s}}$ such that the tableau $\Atilde = \mathcal{P} \, \A \, \mathcal{P}^\tp$ \cref{eqn:permuted-Butcher-tableau} is lower triangular. 
\end{definition}
The calculation of the stage vectors \cref{eqn:GARK-stage}, carried out in the order defined by the permutation matrix $\mathcal{P}$, is done in an implicit-decoupled manner, i.e., any implicit stage calculation involves a single unknown stage vector $Y_i\comp{q}$. There are no nonlinear equations that require solving for multiple stage vectors simultaneously. 

\begin{remark}[Dependency graph for GARK tableau]
	For the GARK scheme \cref{eqn:consistent-Butcher-tableau}, take $\A$ to be the adjacency matrix for a weighted, directed graph $G$.  Cycles in $G$ indicate implicitness, and vertices in a cycle correspond to a stage values that must be solved simultaneously.  Therefore, the IMIM property is equivalent to all cycles in $G$ being loops (cycles with only one edge that start and end at same vertex).
\end{remark}

%%%%%%%%%%%%%%%%%%%%%%%%%%%%%%
\subsection{Implicit-implicit GARK methods with a special coupling  structure}
\label{sec:imim-par-2}
%%%%%%%%%%%%%%%%%%%%%%%%%%%%%%
Of special interest in this work are IMIM-GARK methods where the Butcher tableau \cref{eqn:consistent-Butcher-tableau} \added[id=2]{uses the same coefficients within the lower triangular, diagonal, and upper triangular blocks:}
\comment[id=2]{Replaced ambiguous $\b$ with $\b\D$}
\begin{equation}
\label{eqn:Butcher-tableau-structure-2}
	\begin{butchertableau}{c}
		 \A \\ \hline
		 \b^\tp
	\end{butchertableau}
	~\coloneqq~
	\raisebox{24pt}{$
	\begin{butchertableau}{cccc}
		 \A\D & \A\U & \ldots &  \A\U  \\
		 \A\L &\A\D & \ldots &  \A\U  \\
		 \vdots & \vdots & & \vdots \\
		\A\L & \A\L & \ldots & \A\D \\ \hline
		\b\D* & \b\D* & \ldots &\b\D*
	\end{butchertableau}
	$}.
\end{equation}
\added[id=2]{This structure requires $\A\U, \A\D, \A\L \in \R{s\D}{s\D}$ where $s\D \coloneqq s\comp{1} = \dots s\comp{\nparts}$.}  The blocks in \cref{eqn:Butcher-tableau-structure-2} correspond to explicit and diagonally implicit Runge--Kutta schemes, as follows:
\begin{itemize}
	\item $\A\U$ is a strictly lower triangular matrix ($a_{i,j}\U = 0$ for $i \le j$) corresponding to an explicit Runge--Kutta scheme \added[id=2]{$(\A\U,\b\D,\c\D)$};
	\item $\A\D$ is a lower triangular matrix ($a_{i,j}\D = 0$ for $i < j$) corresponding to a diagonally implicit Runge--Kutta scheme \added[id=2]{$(\A\D,\b\D,\c\D)$}; and
	\item $\A\L$ is a lower triangular matrix ($a_{i,j}\L = 0$ for $i < j$) corresponding to either an explicit or a diagonally implicit Runge--Kutta scheme \added[id=2]{$(\A\L,\b\D,\c\D)$}. 
\end{itemize}
Note that the internal consistency property \cref{eqn:internal-consistency} and the explicit nature of $\A\U$ require that $\A\D$ and $\A\L$ have an explicit first stage (e.g., correspond to ESDIRK schemes).  The stiff accuracy property \cref{eqn:gark-stiff-accuracy} is equivalent to:
\begin{equation*}
	\label{eqn:gark-stiff-accuracy-2}
	\b\D* = \unitvec{s}^\tp \, \A\L = \unitvec{s}^\tp \, \A\D.
\end{equation*}

The method \cref{eqn:Butcher-tableau-structure-2} advances the solution using the following computational process:
\begin{subequations}
	\label{eqn:IMIM-GARK}
	\begin{align}
		\nonumber
		&\texttt{for } i=1,\dots,s\D \texttt{ do:} \\
		\label{eqn:IMIM-GARK-stage}
		&\qquad
		\begin{aligned}
			Y_i\comp{q} &= y_{n} + \Dt \sum_{m < q} \sum_{j=1}^{i} {a}_{i,j}\L \, f\comp{m}_j  
			+\Dt \sum_{j=1}^i {a}_{i,j}\D \, f\comp{q}_j \\
			&\quad + \Dt \sum_{m>q} \sum_{j=1}^{i-1} {a}_{i,j}\U \, f\comp{m}_j, \quad q=1,\dots,\nparts;
		\end{aligned} \\
		%
%		\nonumber
%		& \texttt{endfor} \\
		%
		\label{eqn:IMIM-GARK-solution}
		& y_{n+1} = y_n + \Dt \sum_{q=1}^\nparts \sum_{i=1}^{s\D} b\D_{i} \, f\comp{q}_i.
	\end{align}
\end{subequations}
The method computes the first stage $Y_1\comp{q}$ for all components in the order $q=1,\dots,\nparts$. Then the method computes \deleted[id=1]{each stage  \cref{eqn:IMIM-GARK-stage} computes} $Y_i\comp{q}$, $i \ge 2$, using the previous stages $Y_1\comp{m} \ldots Y_{i-1}\comp{m}$ for all components $m$, and the already computed  stages $Y_i\comp{1} \ldots Y_{i}\comp{q-1}$. The lower triangular structure of $\A\D$ implies that $Y_i\comp{q}$ is computed implicitly in a SDIRK-like manner when $a_{i,i}\D > 0$, and explicitly when  $a_{i,i}\D = 0$.
\ifreport %
\begin{example}[Two-way partitioned IMIM-GARK]
	Consider a method \cref{eqn:IMIM-GARK} with two stages ($s=2$) applied to an autonomous, two-way-partitioned system ($\nparts=2$). The solution is computed as follows:
	\begin{equation*}
		\begin{split}
			Y_1\comp{1} &= Y_1\comp{2} = y_{n}, \\
			\boxed{Y_2\comp{1}} &= y_{n} + \Dt \, {a}_{2,1}\D \, f\comp{1} \mleft( Y_1\comp{1} \mright) + \boxed{\Dt \,{a}_{2,2}\D \, f\comp{1} \mleft( Y_2\comp{1} \mright)} + \Dt\, {a}_{2,1}\U \, f\comp{2} \mleft( Y_1\comp{2} \mright), \\
			\boxed{Y_2\comp{2}} &= y_{n} +\Dt\,{a}_{2,1}\L \, f\comp{1} \mleft( Y_1\comp{1} \mright) +\Dt\,{a}_{2,2}\L \, f\comp{1}\mleft(Y_2\comp{1}\mright) \\
			& \quad + \Dt \,{a}_{2,1}\D \, f\comp{2} \mleft( Y_1\comp{2} \mright) +\boxed{\Dt \,{a}_{2,2}\D \, f\comp{2}\mleft(Y_2\comp{2}\mright)}, \\
			y_{n+1} &= y_n + \Dt\,b_{1} \left( f\comp{1} \mleft( Y_1\comp{1} \mright) + f\comp{2} \mleft( Y_1\comp{2} \mright) \right) \\
			& \quad + \Dt\,b_{2} \left( f\comp{1}\mleft(Y_2\comp{1}\mright) + f\comp{2}\mleft(Y_2\comp{2}\mright) \right).
		\end{split}
	\end{equation*}
	The implicit terms in each stage computation are indicated by boxes.
\end{example}
\fi
\added[id=2]{The order in which stages are evaluated in \cref{eqn:IMIM-GARK} corresponds to the permutation}
\begin{equation} \label{eqn:vec-perm_mat}
	\mathcal{P} = \sum_{i=1}^{\nparts} \unitvec{i,\nparts}^\tp \otimes \eye{s} \otimes \unitvec{i,\nparts}
	 = \sum_{i=1}^{s} \unitvec{i,s} \otimes \eye{\nparts} \otimes \unitvec{i,s}^\tp,
\end{equation}
\added[id=2]{which is known as the vec-permutation matrix \cite{henderson1981vec}.}
\ifreport %
Upon reorganizing the GARK tableau \cref{eqn:consistent-Butcher-tableau} into  \cref{eqn:permuted-Butcher-tableau}, we have the structure
\begin{subequations}
	\label{eqn:imim-permuted-Butcher}
	\begin{equation}
	\label{eqn:imim-permuted-Butcher-tableau}
		\begin{butchertableau}{c|cccc}
			\ctilde\comp{1} & \Atilde\comp{1,1} & \Atilde\comp{1,2} & \ldots & \Atilde\comp{1,s} \\
			\ctilde\comp{2} & \Atilde\comp{2,1} & \Atilde\comp{2,2} & \ldots & \Atilde\comp{2,s} \\
			\vdots &  \vdots & \vdots & & \vdots \\
			\ctilde\comp{s} &  \Atilde\comp{s,1} & \Atilde\comp{s,2} & \ldots & \Atilde\comp{s,s} \\ \hline
			& \btilde\comp{1}* & \btilde\comp{2}* & \ldots & \btilde\comp{s}*
		\end{butchertableau},
	\end{equation}
	with blocks $\Atilde\comp{i,j} \in \R{\nparts}{\nparts}$ and  $\btilde\comp{i}, \ctilde\comp{i} \in \R{\nparts}$ defined as
	\begin{equation}
		\label{eqn:imim-permuted-Butcher-blocks}
		\Atilde\comp{i,j} =
		\begin{smatrix}[1.3]
			a\D_{i,j} & a\U_{i,j} & \ldots & a\U_{i,j} \\
			a\L_{i,j} &a\D_{i,j} & \ldots & a\U_{i,j} \\
			\vdots & \vdots & & \vdots \\
			a\L_{i,j} & a\L_{i,j} & \ldots & a\D_{i,j} 
		\end{smatrix},
		\quad
		\btilde\comp{i} = \b_i\,\one{\nparts},
		\quad
		\ctilde\comp{i} = \c_i\,\one{\nparts}.
	\end{equation}
\end{subequations}
We see that $\Atilde\comp{i,j} = \zero{\nparts}{\nparts}$ for $j > i$, and that $\Atilde\comp{i,i}$ are lower triangular matrices, with equal diagonal entries $a\I_{i,i}$. Therefore \cref{eqn:imim-permuted-Butcher} it is an IMIM-GARK scheme according to \Cref{def:imim-gark}.
\fi

We consider two subclasses of IMIM-GARK schemes, ADI-GARK and parallel ADI-GARK, for their practical appeal.

\begin{definition}[ADI-GARK]
    \label{definition:ADI-GARk}
	An Alternating Direction  Implicit GARK (ADI-GARK) scheme is an IMIM-GARK method with the structure \cref{eqn:Butcher-tableau-structure-2} where $\A \L = \A \D = \A\I$, $\A \U = \A\E$, $\b\D = \b\I$, and $\c\D = \c\I$.
\end{definition}
We note that an ADI-GARK method is stiffly accurate \cref{eqn:gark-stiff-accuracy} iff the implicit component $(\A\I,\b\I,\c\I)$ is stiffly accurate in the Runge--Kutta sense. 

\ifreport
\begin{remark}
\label{rem:esdirk-blocks}
When the implicit base method is an ESDIRK scheme with $\A\I_{i,i} = \gamma$ for $i \ge 2$, the Butcher tableau \cref{eqn:imim-permuted-Butcher} of the corresponding ADI-GARK method is:
\begin{align*}
    \Ahat\comp{i,j} \big|_{i < j} &= \zero{\nparts}{\nparts}, \quad
    \Ahat\comp{1,1} = \zero{\nparts}{\nparts}, \quad
    \Ahat\comp{i,i}\big|_{i \ge 2} = 
    \Gamma = \begin{bmatrix}
    	\gamma & \dots & 0 \\ 
		\vdots  & \ddots & \vdots \\ \gamma  & \dots & \gamma
	\end{bmatrix}.
\end{align*}
\end{remark}
\fi

\begin{definition}[Parallel ADI-GARK]
      \label{definition:P-ADI-GARk}
A parallel ADI-GARK scheme is an IMIM-GARK method with the structure \cref{eqn:Butcher-tableau-structure-2} where $\A \L = \A \U = \A\E$, $\A\D = \A\I$, $\b\D = \b\I$, and $\c\D = \c\I$. The nonlinear systems for stages $Y_i\comp{q}$ \cref{eqn:IMIM-GARK-stage} are solved in parallel for a given $i$ for all components $q=1,\dots,\nparts$.
\end{definition}
\begin{remark} \label{rem:serial_parallel}
    Owing to their structure, the order conditions of ADI-GARK and parallel ADI-GARK schemes are much simpler than the general GARK conditions \cref{eqn:GARK-order-conditions} \cite{Sandu_2015_GARK}. These order conditions depend on the underlying implicit and explicit methods and are the same for both ADI-GARK and parallel ADI-GARK schemes; the difference between the methods is given by the way these blocks are assembled in the Butcher tableau \cref{eqn:Butcher-tableau-structure-2}.
\end{remark}

In \cref{sec:new_methods}, we provide the coefficients of new ADI-GARK and parallel ADI-GARK schemes of orders three and four.
	%%%%%%%%%%%%%%%%%%%%%
\section{\added[id=2]{Linear stability analysis}}\comment[id=2]{This section has been added as requested by reviewer 2}
\label{sec:Linear-stability}
%%%%%%%%%%%%%%%%%%%%%

Consider an IMIM-GARK method \cref{eqn:Butcher-tableau-structure-2} and the corresponding permuted tableau \cref{eqn:permuted-Butcher-tableau} using the permutation matrix \cref{eqn:vec-perm_mat}.  We apply this method to the linear, partitioned system
\begin{equation}
	y' = \sum_{m=1}^{\nparts} \lambda\comp{m} \, y,
\end{equation}
where $\lambda\comp{m} \in \Cplx$.  We employ the helpful notation
\begin{equation*}
	\begin{alignedat}{2}
		\mathbf{z} & \coloneqq \begin{bmatrix}
			z\comp{1} & \ldots & z\comp{\nparts}
		\end{bmatrix}^\tp,
		\qquad &
		\mathbf{\widehat{Z}} & \coloneqq \diag\left( z\comp{1} , \dots , z\comp{\nparts} \right), \\
		\mathbf{Z}_{s} &\coloneqq \mathbf{\widehat{Z}} \otimes \eye{s}, \qquad & \mathbf{\widetilde{Z}}_{s} &\coloneqq \eye{s} \otimes \mathbf{\widehat{Z}},
		% \quad   z := z\comp{1} + \dots + z\comp{\nparts},
	\end{alignedat}
\end{equation*}
where $z\comp{m} \coloneqq h \, \lambda\comp{m}$. From \cite{Sandu_2015_GARK}, it is known that
\begin{equation*}
	y_{n+1} = R\mleft( \mathbf{z} \mright) \, y_n,
\end{equation*}
where the stability function $R\mleft( \mathbf{z} \mright)$ can be written compactly as
\begin{equation}
	\label{eqn:stability-function}
	\begin{split}
		R\mleft( \mathbf{z} \mright) &= 1 + \b^\tp \, \mathbf{Z}_s \left( \eye{\nparts s} - \A \, \mathbf{Z}_s \right)^{-1} \one{\nparts s} \\
		&= 1 + {\b}^\tp \left( \mathbf{Z}_s^{-1} - \A \right)^{-1} \one{\nparts s} \\
		&= 1 + \btilde^\tp \left( \mathbf{\widetilde{Z}}_s^{-1} - \Atilde \right)^{-1} \one{\nparts s}.
	\end{split}
\end{equation}
For stiffly accurate methods, this simplifies to
\begin{equation*}
%	\label{eqn:stab-stiffly-accurate}
%	\begin{split}
		R\mleft( \mathbf{z} \mright) %&= 1 + \unitvec{\nparts s}^\tp  \Atilde\, \left( \mathbf{\widetilde{Z}}^{-1} -   \Atilde \right)^{-1} \one{\nparts s} \\
		%&=& 1 +   \unitvec{\nparts s}^\tp \big( \Atilde- \mathbf{\widetilde{Z}}^{-1}+ \mathbf{\widetilde{Z}}^{-1} \big)\, \left( \mathbf{\widetilde{Z}}^{-1} -   \Atilde \right)^{-1} \one{\nparts s} \\
		%&=& \unitvec{\nparts s}^\tp  \mathbf{\widetilde{Z}}^{-1} \left( \mathbf{\widetilde{Z}}^{-1} -   \Atilde \right)^{-1} \one{\nparts s}\\
%		&=  \unitvec{\nparts s}^\tp  \left( \eye{\nparts s} -   \Atilde\mathbf{\widetilde{Z}} \right)^{-1}  \one{\nparts s}.
		 =  \unitvec{\nparts s}^\tp   \left( \eye{\nparts s} -   \A \, \mathbf{Z}_s \right)^{-1}  \one{\nparts s},
%	\end{split}
\end{equation*}
and in the case where $\A$ is invertible, it holds that \cite{Sandu_2015_GARK}
\begin{equation} \label{eqn:stiff-acc-lim}
	\lim_{z\comp{\nparts} \to -\infty} R(\mathbf{z}) = 1 - \btilde^\tp  \Atilde^{-1} \one{\nparts s} =  1 - \b^\tp  \A^{-1} \one{\nparts s}.
\end{equation}

These results hold for the broad class of IMIM-GARK methods and the many special subclasses.  Further stability results for GARK methods applied to scalar and $2 \times 2$ linear test problems can be found in \cite{Sandu_2020_MR-GARK_Implicit}.  Notably, the decoupled nature of IMIM-GARK methods prevents them from achieving A-stability for $2 \times 2$, linear ODEs \cite[Theorem 3.4]{Sandu_2020_MR-GARK_Implicit}.

	%!TEX root = ../main.tex
%%%%%%%%%%%%%%%%%%%%%%%%%%%%%%%%%%%%%%%%%%%%%%%%%%%%%%%%%%%%%%%
\section{\added[id=2]{Classical LOD and ADI} methods in the GARK framework}
\label{sec:classical-adi-methods}
%%%%%%%%%%%%%%%%%%%%%%%%%%%%%%%%%%%%%%%%%%%%%%%%%%%%%%%%%%%%%%%

This section is devoted to showing that classical Locally One-Dimensional (LOD) and Alternating Direction Implicit (ADI) schemes appearing in the literature can be formulated within the IMIM-GARK framework. In particular, the consistency of such methods can be treated in a unified way within the GARK formalism. To this end we consider the non-autonomous ODE problem \cref{eqn:additive-ode}.  We shall deal with GARK methods of the form \cref{eqn:consistent-Butcher-tableau}. The consistency order of the following methods can be checked by using the order conditions indicated in \cref{app:order_conditions} at the end of the manuscript.

%%%%%%%%%%%%%%%%%%
\subsection{LOD-Backward Euler method}
%%%%%%%%%%%%%%%%%%
%
The LOD-Backward Euler method \cite[page 348]{Hundsdorfer_2003_book} reads
\begin{equation}
	\label{eqn:ADI-BE}
	\begin{split}
		v_0 &= y_n,\\
		v_{q} &= v_{q-1}+\Dt\, f\comp{q}(t_{n+1},v_{q}), \qquad q = 1,\ldots,\nparts,\\
		y_{n+1} &= v_\nparts.
	\end{split}
\end{equation}
We rewrite \cref{eqn:ADI-BE} in GARK notation with the stages $Y\comp{q} \coloneqq v_{q}$:
\begin{equation}
\label{eqn:ADI-BE-gark}
	\begin{split}
		Y\comp{q} &=y_n + \Dt \sum_{m=1}^{q} f\comp{m} \mleft( t_{n+1}, Y\comp{m} \mright),  \quad q = 1,\ldots,\nparts,\\
		y_{n+1}&=Y\comp{\nparts} = y_n + \Dt \sum_{m=1}^{\nparts} f\comp{m}\mleft(t_{n+1},Y\comp{m}\mright).
	\end{split}
\end{equation}
This method has order of consistency one, and it is stiffly accurate but is not internally consistent.  It has the form \cref{eqn:Butcher-tableau-structure-2} and further is an ADI-GARK method with the coefficients
\begin{alignat*}{2}
	\A\comp{q,m} &= \A\I = \begin{bmatrix}1\end{bmatrix}, \text{ for } m \leq q, & \qquad \A\comp{q,m} &= \A\E = \begin{bmatrix}0\end{bmatrix}, \text{ for } m > q, \\
	\b\comp{q} &= \b\I= \begin{bmatrix}1\end{bmatrix}, & \qquad \c\comp{q} &= \c\I = \begin{bmatrix}1\end{bmatrix},
\end{alignat*}
for $q,m = 1, \dots, \nparts$. 

%%%%%%%%%%%%%%%%%%
\subsection{Yanenko's LOD-Crank-Nicolson method}
%%%%%%%%%%%%%%%%%%
%
Yanenko's LOD-Crank-Nicolson method \cite[page 351]{Hundsdorfer_2003_book} reads
\begin{equation}
	\label{eqn:ADI-CN}
	\begin{split}
		v_0 &= y_n,\\
		v_{q}&= v_{q-1}+\frac{\Dt}{2} \left( f\comp{q}(t_{n}+c_{q-1}\,\Dt,v_{q-1})+ f\comp{q}(t_{n}+c_{q}\,\Dt,v_{q}) \right),\\
		&  \qquad q = 1,\ldots,\nparts,\\
		y_{n+1}&= v_\nparts,
	\end{split}
\end{equation}
with $c_0=0$, $c_\nparts=1$, and $c_{q}=1/2$ for $q = 1, \dots, \nparts - 1$. 
Iterating the stage values in \cref{eqn:ADI-CN} leads to the following formula:
\begin{equation*}
	\begin{split}
		v_q&= y_n+ \frac{\Dt}{2} \sum_{m=1}^q \left( f\comp{m}\mleft( t_{n}+c_{m-1}\,\Dt,v_{m-1} \mright)
		+ f\comp{m} \mleft( t_{n}+c_{m}\,\Dt,v_{m} \mright) \right).
	\end{split}
\end{equation*}
The method is cast in the GARK form \cref{eqn:consistent-Butcher-tableau} by defining the stages
\begin{equation*}
	Y\comp{1} \coloneqq \begin{bmatrix}v_0\\v_{1}\end{bmatrix}, \quad \ldots \quad
	Y\comp{q} \coloneqq \begin{bmatrix}v_{q-1}\\v_q\end{bmatrix}, \quad \ldots \quad 
	Y\comp{\nparts} \coloneqq \begin{bmatrix}v_{\nparts-1}\\v_{\nparts}\end{bmatrix}.
\end{equation*} 
Note that $Y\comp{q}_1=Y\comp{q-1}_2$ for $q \ge 2$. We rewrite \cref{eqn:ADI-CN} in GARK notation as follows:
\begin{equation}
	\label{eqn:ADI-CN-GARK}
	\begin{split}
		Y\comp{q}_1 &= y_n+ \frac{\Dt}{2} \sum_{m=1}^{q-1} \left( f\comp{m} \mleft( t_{n}+c_{m-1} \, \Dt,Y\comp{m}_1 \mright)
		+ f\comp{m} \mleft( t_{n}+c_{m}\,\Dt,Y\comp{m}_2 \mright) \right), \\
		Y\comp{q}_2 &= y_n+ \frac{\Dt}{2} \sum_{m=1}^q \left( f\comp{m} \mleft( t_{n}+c_{m-1}\,\Dt, Y\comp{m}_1 \mright)
		+ f\comp{m} \mleft(t_{n}+c_{m}\,\Dt,Y\comp{m}_2 \mright) \right), \\
		y_{n+1}&= Y\comp{\nparts}_2.
	\end{split}
\end{equation}
The GARK Butcher tableau \cref{eqn:Butcher-tableau-structure-2} is defined by the coefficient matrices:
\begin{alignat*}{2}
	\A\comp{q,m}&=\A\L=\begin{smatrix}\frac{1}{2} & \frac{1}{2} \\\frac{1}{2} & \frac{1}{2}\end{smatrix}, \text{ for } m < q,
	&\qquad
	\A\comp{q,q}&=\A\D=\begin{smatrix}0 & 0 \\\frac{1}{2} & \frac{1}{2}\end{smatrix}, \\
	\A\comp{q,m}&=\A\U=\begin{smatrix}0 & 0 \\0 & 0\end{smatrix}, \text{ for } m > q,
	& \qquad \b\comp{q}&=\b\D= \begin{smatrix} \frac{1}{2} \\ \frac{1}{2} \end{smatrix}, \\
	\c\comp{q}&=\c\D=\begin{smatrix} c_{q-1} \\ c_{q} \end{smatrix},
	& \qquad m, q &= 1, \dots, \nparts.
\end{alignat*}
We then immediately observe that \cref{eqn:ADI-CN-GARK} is only order one. Furthermore, it is stiffly accurate and not internally consistent. On the other hand, the matrix $\A$ is already in lower triangular form since the stages values $Y_i\comp{q}$ are displayed in the same way they are actually computed.

\begin{remark}
	Symmetric and parallel versions of the method \cref{eqn:ADI-CN} are known to provide respective second order methods at the expense of doubling the computational cost \cite[page 351-352]{Hundsdorfer_2003_book}. Both methods can also be expressed as GARK methods. First, the symmetric Yanenko's LOD-Crank-Nicolson method \cite[page 351]{Hundsdorfer_2003_book} is given by
	\begin{equation}
		\label{eqn:symmetric-ADI-CN}
		\begin{aligned}
			v_0&= y_n,\\
			v_{q}&= v_{q-1}+\frac{\Dt}{4}\left( f\comp{q} \mleft( t_{n}+ \frac{c_{q-1}}{2} \, \Dt,v_{q-1})+ f\comp{q}(t_{n}+ \frac{c_{q}}{2} \, \Dt,v_{q} \mright) \right),
			\\
			& \qquad q = 1, \dots, \nparts,
			\\
			w_0 &= v_\nparts,
			\\
			w_{q}&= w_{q-1}+\frac{\Dt}{4} \left( f\comp{\nparts+1-q} \mleft( t_{n}+\frac{1 + c_{q-1}}{2} \, \Dt,w_{q-1} \mright) \right.
			\\
			& \quad \left. + f\comp{\nparts+1-q} \mleft(t_{n}+\frac{1 + c_{q}}{2} \, \Dt,w_{q} \mright) \right), \qquad q = 1, \dots, \nparts,
			\\
			y_{n+1}&= w_\nparts.
		\end{aligned}
	\end{equation}
	\added[id=2]{This computational process is mapped to GARK stages by}
	\begin{equation} \label{eqn:symmetric-ADI-CN_stages}
    	Y\comp{1} \coloneqq \begin{bmatrix}v_0\\v_{1}\\w_{\nparts-1}\\w_{\nparts}\end{bmatrix}, \quad \ldots \quad
    	Y\comp{q} \coloneqq \begin{bmatrix}v_{q-1}\\v_q\\w_{\nparts-q}\\w_{\nparts+1-q}\end{bmatrix}, \quad \ldots \quad 
    	Y\comp{\nparts} \coloneqq \begin{bmatrix}v_{\nparts-1}\\v_{\nparts}\\w_{0}\\w_{1}\end{bmatrix}.
    \end{equation} 
	Now, \cref{eqn:symmetric-ADI-CN} can be formulated in the GARK framework as follows:
	\comment[id=2]{Fixed the indexing of stages 3 and 4}
	\begin{equation}
		\label{eqn:symmetric-ADI-CN-GARK}
		\begin{aligned}
			Y\comp{q}_1 &=  y_n+ \frac{\Dt}{4} \sum_{m=1}^{q-1} \left( f\comp{m} \mleft( t_{n}+\frac{c_{m-1}}{2} \, \Dt,Y\comp{m}_1 \mright) + f\comp{m} \mleft(t_{n}+\frac{c_{m}}{2} \, \Dt,Y\comp{m}_2 \mright) \right), 
			\\
			Y\comp{q}_2 &=  Y\comp{q+1}_1, \qquad q = 1, \dots, \nparts,
			\\
			Y\comp{q}_3 &=  y_n + \frac{\Dt}{4} \sum_{m=1}^{\nparts} \left( f\comp{m} \mleft( t_{n} + \frac{c_{m-1}}{2} \, \Dt, Y\comp{m}_1 \mright) + f\comp{m} \mleft( t_{n} + \frac{c_{m}}{2} \, \Dt, Y\comp{m}_2 \mright) \right) \\
			& \quad + \frac{\Dt}{4} \sum_{m=q+1}^{\nparts} \left( f\comp{m} \mleft( t_{n}+\frac{1+c_{\nparts-m}}{2} \, \Dt, Y\comp{m}_3 \mright) \right. \\
			& \quad \left. + f\comp{m} \mleft( t_{n}+\frac{1+c_{\nparts+1-m}}{2} \, \Dt,Y\comp{m}_4 \mright) \right), \\
			Y\comp{q}_4 &=  Y\comp{q-1}_3, \qquad q = \nparts, \dots, 1,
			\\
			y_{n+1}&= Y\comp{1}_4.
		\end{aligned}
	\end{equation}
	The corresponding coefficient matrices are then defined by
	\begin{alignat*}{2}
		\A\comp{q,m}&= \A\L&=&\begin{smatrix} \frac{1}{4} & \frac{1}{4} & 0 & 0 \\\frac{1}{4} & \frac{1}{4} & 0 & 0 \\
		\frac{1}{4} & \frac{1}{4} & 0 & 0 \\
		\frac{1}{4} & \frac{1}{4} & 0 & 0 
		\end{smatrix}, \text{ for } m < q, \\
		\A\comp{q,q}&= \A\D&=&\begin{smatrix} 0 & 0 & 0 & 0 \\\frac{1}{4} & \frac{1}{4} & 0 & 0 \\
		\frac{1}{4} & \frac{1}{4} & 0 & 0 \\
		\frac{1}{4} & \frac{1}{4} & \frac{1}{4} & \frac{1}{4}
		\end{smatrix}, \\
		\A\comp{q,m}&= \A\U&=&\begin{smatrix}
		0 & 0 & 0 & 0 \\0 & 0 & 0 & 0 \\
		\frac{1}{4} & \frac{1}{4} & \frac{1}{4} & \frac{1}{4} \\
		\frac{1}{4} & \frac{1}{4} & \frac{1}{4} & \frac{1}{4}
		\end{smatrix}, \text{ for } m > q, \\
		\b\comp{q}* &= \b\D* &=& \begin{bmatrix} \frac{1}{4} & \frac{1}{4} & \frac{1}{4} & \frac{1}{4} \end{bmatrix}, \\
		\c\comp{q}* &= \c\D* &=& \begin{bmatrix} \frac{c_{q-1}}{2} & \frac{c_{q}}{2} & \frac{1+c_{\nparts-q}}{2} & \frac{1+c_{\nparts+1-q}}{2} \end{bmatrix},
	\end{alignat*}
	for $q,m = 1, \dots, \nparts$.  Now it is readily checked that \cref{eqn:symmetric-ADI-CN-GARK} reaches order two. Furthermore, observe that the method expressed this way does not satisfy the stiff accuracy condition \cref{eqn:gark-stiff-accuracy}. However, the stages can be reordered into a lower triangular form that is stiffly accurate.
	\added[id=2]{The permuted tableau \cref{eqn:permuted-Butcher-tableau} is given by}
	\begin{equation*}
		\begin{butchertableau}{c|ccccc:ccccc}
			& f\comp{1}_1 & f\comp{1}_2 & \ldots & f\comp{\nparts}_1 & f\comp{\nparts}_2  & f\comp{\nparts}_3 & f\comp{\nparts}_4  & \ldots & f\comp{1}_3 & f\comp{1}_4 \\ \hline
			
			Y\comp{1}_1 & 0 & 0 &  \ldots & 0 & 0 & 0 & 0 &  \ldots & 0 & 0
			\\ 
			
			Y\comp{1}_2 &\frac{1}{4} & \frac{1}{4} &  \ldots & 0 & 0 & 0 & 0 &  \ldots & 0 & 0
			\\
			
			\vdots &\vdots& \vdots &  \ddots &  \vdots & \vdots & \vdots & \vdots &  \ddots &  \vdots  & \vdots 
			\\
			
			Y\comp{\nparts}_1 &\frac{1}{4} & \frac{1}{4} &  \ldots & 0 & 0 & 0 & 0 &  \ldots & 0 & 0
			\\ 
			Y\comp{\nparts}_2 & \frac{1}{4} & \frac{1}{4} &  \ldots & \frac{1}{4} & \frac{1}{4} & 0 & 0 &  \ldots & 0 & 0
			\\\hdashline 
			Y\comp{\nparts}_3 & \frac{1}{4} & \frac{1}{4} &  \ldots &  \frac{1}{4} & \frac{1}{4} & 0 & 0 &  \ldots & 0 & 0
			\\ 
			
			Y\comp{\nparts}_4 &\frac{1}{4} & \frac{1}{4} &  \ldots &  \frac{1}{4} & \frac{1}{4} & \frac{1}{4} & \frac{1}{4} &  \ldots & 0 & 0
			\\ 
			\vdots &\vdots& \vdots &  \ddots & \vdots & \vdots & \vdots & \vdots &  \ddots &  \vdots & \vdots
			\\ 
			Y\comp{1}_3 & \frac{1}{4} & \frac{1}{4} &  \ldots &  \frac{1}{4} & \frac{1}{4} & \frac{1}{4} & \frac{1}{4} &  \ldots &  0 & 0
			
			\\ 
			Y\comp{1}_4 & \frac{1}{4} & \frac{1}{4} &  \ldots &  \frac{1}{4} & \frac{1}{4} & \frac{1}{4} & \frac{1}{4} &  \ldots &  \frac{1}{4} & \frac{1}{4}
			\\\hline 
			y_{n+1} &\frac{1}{4} & \frac{1}{4} &  \ldots &  \frac{1}{4} & \frac{1}{4}& \frac{1}{4} & \frac{1}{4} &  \ldots &  \frac{1}{4} & \frac{1}{4}
		\end{butchertableau},
	\end{equation*}
	\added[id=2]{and in this form, the stiff accuracy is evident.}
	
	Secondly, for the parallel Yanenko's LOD-Crank-Nicolson method \cite[page 352]{Hundsdorfer_2003_book}
	\begin{equation}
		\label{eqn:parallel-ADI-CN}
		\begin{split}
			v_0 &= y_n,\\
			v_{q} &= v_{q-1}+\frac{\Dt}{2}\left( f\comp{q} \mleft( t_{n}+ c_{q-1} \, \Dt,v_{q-1} \mright)+ f\comp{q} \mleft( t_{n}+ c_{q} \, \Dt,v_{q} \mright) \right), \\
			& \qquad q = 1,\ldots,\nparts,\\
			w_0 &= y_n,\\
			w_{q}&= w_{q-1}+\frac{\Dt}{2} \left(f\comp{\nparts+1-q} \mleft( t_{n}+ c_{q-1} \, \Dt,w_{q-1} \mright) + f\comp{\nparts+1-q} \mleft( t_{n}+ c_{q} \, \Dt,w_{q} \mright) \right), \\
			& \qquad q = 1,\ldots,\nparts,\\
			y_{n+1}&= \frac{1}{2}(v_\nparts+w_\nparts),
		\end{split}
	\end{equation}
	we similarly consider the stages \cref{eqn:symmetric-ADI-CN_stages} and rewrite \cref{eqn:parallel-ADI-CN} in GARK notation as
	\comment[id=2]{Fixed the indexing of stages 3 and 4}
	\begin{equation}
		\label{eqn:parallel-ADI-CN-GARK}
		\begin{split}
			Y\comp{q}_1 &=  y_n+ \frac{\Dt}{2} \sum_{m=1}^{q-1} \left( f\comp{m} \mleft( t_{n}+c_{m-1} \, \Dt,Y\comp{m}_1 \mright) + f\comp{m} \mleft(t_{n}+c_{m} \, \Dt,Y\comp{m}_2 \mright) \right), 
			\\
			Y\comp{q}_2 &=  Y\comp{q+1}_1, \qquad q = 1, \dots, \nparts,
			\\
			Y\comp{q}_3 &=  y_n + \frac{\Dt}{4} \sum_{m=q+1}^{\nparts} \left( f\comp{m} \mleft( t_{n}+c_{\nparts-m} \, \Dt, Y\comp{m}_3 \mright) \right. \\
			& \quad \left. + f\comp{m} \mleft( t_{n}+ c_{\nparts+1-m} \, \Dt,Y\comp{m}_4 \mright) \right), \\
			Y\comp{q}_4 &=  Y\comp{q-1}_3, \qquad q = \nparts, \dots, 1,
			\\
			y_{n+1}&= \frac{1}{2} \left(Y\comp{\nparts}_2 + Y\comp{1}_4 \right).
		\end{split}
	\end{equation}
	Its GARK Butcher tableau \cref{eqn:Butcher-tableau-structure-2} is then defined by the following coefficient matrices:
	\begin{alignat*}{2}
		\A\comp{q,m}&=\A\L&=&\begin{smatrix}
		\frac{1}{2} & \frac{1}{2} & 0 & 0 \\
		\frac{1}{2} & \frac{1}{2} & 0 & 0 \\
		0 & 0 & 0 & 0 \\
		0 & 0 & 0 & 0 
		\end{smatrix}, \text{ for } m < q, \\
		\A\comp{q,q}&=\A\D&=&\begin{smatrix}
		0 & 0 & 0 & 0 \\
		\frac{1}{2} & \frac{1}{2} & 0 & 0 \\
		0 & 0 & 0 & 0 \\
		0 & 0 & \frac{1}{2} & \frac{1}{2}
		\end{smatrix}, \\
		\A\comp{q,m}&=\A\U&=&\begin{smatrix}
		0 & 0 & 0 & 0 \\
		0 & 0 & 0 & 0 \\
		0 & 0 & \frac{1}{2} & \frac{1}{2} \\
		0 & 0 & \frac{1}{2} & \frac{1}{2}
		\end{smatrix}, \text{ for } m > q, \\
		\b\comp{q}* &= \b\D &=& \begin{bmatrix} \frac{1}{4} & \frac{1}{4} & \frac{1}{4} & \frac{1}{4} \end{bmatrix}, \\
		\c\comp{q}* &= \c\D* &=& \begin{bmatrix} c_{q-1} & c_{q} & c_{\nparts-q} & c_{\nparts+1-q} \end{bmatrix}.
	\end{alignat*}
	From here, it is readily checked that \cref{eqn:parallel-ADI-CN-GARK} reaches order two. 

\end{remark}

%%%%%%%%%%%%%%%%%%
\subsection{Trapezoidal Splitting}
%%%%%%%%%%%%%%%%%%
%
The Trapezoidal Splitting scheme \cite[page 359]{Hundsdorfer_2003_book} reads
\begin{equation}
	\label{eqn:splitting-Trapezoidal}
	\begin{split}
		v_0&= y_n,\\
		v_{q}&= v_{q-1}+\frac{\Dt}{2} \, f\comp{q} \mleft( t_{n},v_{q-1} \mright),\qquad q = 1,\ldots,\nparts,\\
		v_{\nparts+q}&= v_{\nparts+q-1}+\frac{\Dt}{2} \, f\comp{\nparts+1-q}(t_{n+1},v_{\nparts+q}),\qquad q = 1,\ldots,\nparts,\\
		y_{n+1}&= v_{2\nparts}.
	\end{split}
\end{equation}
The method is cast in the GARK form \cref{eqn:consistent-Butcher-tableau} by defining the stages
\comment[id=2]{Fix index for $Y\comp{\nparts}$}
\begin{equation*}
	Y\comp{1} \coloneqq \begin{bmatrix}v_0\\v_{2\nparts}\end{bmatrix}, \quad \ldots \quad
	Y\comp{q} \coloneqq \begin{bmatrix}v_{q-1} \\ v_{2\nparts-q+1}\end{bmatrix}, \quad \ldots \quad
	Y\comp{\nparts} \coloneqq \begin{bmatrix}v_{\nparts-1}\\v_{\nparts+1}\end{bmatrix}.
\end{equation*}
Iterating the stages \cref{eqn:splitting-Trapezoidal} leads to
\begin{equation}
	\label{eqn:splitting-Trapezoidal-GARK}
	\begin{split}
		Y\comp{q}_1 &= y_n+ \frac{\Dt}{2} \sum_{m=1}^{q-1} f\comp{m} \mleft( t_{n},Y\comp{m}_1 \mright) ,\\
		Y\comp{q}_2&= y_n + \frac{\Dt}{2} \sum_{m=1}^{\nparts} f\comp{m} \mleft( t_{n}, Y\comp{m}_1 \mright)+ \frac{\Dt}{2} 
		\sum_{m=q}^{\nparts} f\comp{m} \mleft( t_{n+1},Y\comp{m}_2 \mright), \\
		& \qquad q = \nparts, \ldots, 1\\
		y_{n+1}&= Y\comp{1}_2,
		% y_n+\sum_{\nu=1}^{\nparts} \frac{\Dt}{2} f\comp{\nu}\bigl(t_{n},Y\comp{\nu}_1\bigr)+\sum_{\nu=1}^{\nparts} \frac{\Dt}{2} f\comp{\nu}\bigl(t_{n+1},Y\comp{\nu}_2\bigr).
\end{split}
\end{equation}
and the coefficient matrices
\begin{alignat*}{2}
	\A\comp{q,m}&=\A\L=\begin{smatrix}\frac{1}{2} & 0 \\\frac{1}{2} & 0\end{smatrix}, \text{ for } m < q,
	&\qquad
	\b\comp{q}&= \b\D = \begin{smatrix} \frac{1}{2} \\ \frac{1}{2} \end{smatrix}, \\
	\A\comp{q,m}&=\A\D=\A\U=\begin{smatrix}0 & 0 \\\frac{1}{2} & \frac{1}{2}\end{smatrix}, \text{ for } m \geq q,
	& \qquad \c\comp{q} &= \c\D = \begin{smatrix} 0 \\ 1 \end{smatrix}.
\end{alignat*}
From \cref{eqn:GARK-order-conditions}, it is observed that \cref{eqn:splitting-Trapezoidal-GARK} is a second order method and is not internally consistent.
\added[id=2]{The method \cref{eqn:splitting-Trapezoidal-GARK} also admits the following lower triangular formulation where the stage values are presented in the way they are  solved:}
\begin{equation*}
	\begin{butchertableau}{c|cccc:cccc}
		& f\comp{1}_1 & f\comp{2}_1 & \ldots & f\comp{\nparts}_1  & f\comp{\nparts}_2 & f\comp{\nparts-1}_2  & \ldots & f\comp{1}_2\\ \hline
		Y\comp{1}_1 & 0 & 0 &  \ldots & 0 & 0 & 0 &  \ldots & 0 \\
		Y\comp{2}_1 &\frac{1}{2} & 0 &  \ldots & 0 & 0 & 0 &  \ldots & 0\\
		\vdots &\vdots& \vdots &  \ldots &  \vdots & \vdots & \vdots &  \ldots &  \vdots
		\\ Y\comp{\nparts}_1 &\frac{1}{2} & \frac{1}{2} &  \ldots & 0 & 0 & 0 &  \ldots & 0 \\\hdashline
		Y\comp{\nparts}_2 &\frac{1}{2} & \frac{1}{2} &  \ldots &  \frac{1}{2} & \frac{1}{2} & 0 &  \ldots & 0
		\\ Y\comp{\nparts-1}_2 &\frac{1}{2} & \frac{1}{2} &  \ldots &  \frac{1}{2} & \frac{1}{2} & \frac{1}{2} &  \ldots & 0
		\\ \vdots &\vdots& \vdots &  \ldots &  \vdots & \vdots & \vdots &  \ldots &  \vdots
		\\ Y\comp{1}_2 &\frac{1}{2} & \frac{1}{2} &  \ldots &  \frac{1}{2} & \frac{1}{2} & \frac{1}{2} &  \ldots &  \frac{1}{2}\\\hline
		y_{n+1} & \frac{1}{2} & \frac{1}{2} &  \ldots &  \frac{1}{2} & \frac{1}{2} & \frac{1}{2} &  \ldots &  \frac{1}{2}
	\end{butchertableau}
\end{equation*}
\added[id=2]{In this permuted form, the stiff accuracy condition holds.}

%%%%%%%%%%%%%%%%%%
\subsection{Douglas method}
%%%%%%%%%%%%%%%%%%
%
We now consider the partitioned system
\begin{equation}
	\label{eqn:split-with-nonstiff}
	y'=f(t,y)=\sum_{m=0}^\nparts f\comp{m}(t,y).
\end{equation}
Compared to \cref{eqn:additive-ode}, there is the additional partition $f\comp{0}$, which is nonstiff and can be treated explicitly.  The Douglas splitting scheme \cite[page 373]{Hundsdorfer_2003_book} applied to \cref{eqn:split-with-nonstiff} reads
\begin{equation}
	\label{eqn:splitting-Douglas}
	\begin{split}
		v_0&= y_n+\Dt\, f(t_n,y_n),\\
		v_{q}&= v_{q-1}+\theta\,\Dt\left( f\comp{q}(t_{n+1},v_{q})- f\comp{q}(t_{n},y_n)\right),\qquad q = 1,\ldots,\nparts,\\
		y_{n+1}&= v_{\nparts}.
	\end{split}
\end{equation}
Note that an explicit Euler step for the entire system is followed by stabilizing corrections, which are implicit stages in only one component at a time. The method \cref{eqn:splitting-Douglas} is cast in the GARK formulation by defining the stages
\comment[id=2]{Fixed the summation indexing}
\begin{align*}
	Y\comp{q}_1 &= y_n, \qquad q = 0, \dots, \nparts, \\
	Y\comp{q}_2 &= y_n + \Dt \, f\comp{0} \mleft(t_{n}, Y\comp{0}_1 \mright) \\
	&\quad + (1 - \theta) \, \Dt \sum_{m=1}^q f\comp{m} \mleft(t_{n}, Y\comp{m}_1 \mright) + \Dt \sum_{m=q+1}^{\nparts} f\comp{m} \mleft(t_{n}, Y\comp{m}_1 \mright) \\
	& \quad + \theta \, \Dt \sum_{m=1}^q f\comp{m} \mleft(t_{n+1}, Y\comp{m}_2 \mright),\qquad q = 1, \dots, \nparts, \\
	y_{n+1} &= Y\comp{\nparts}_2.
\end{align*}
From here, we can see this corresponds to an ADI-GARK scheme with coefficient matrices
\begin{subequations} \label{eqn:douglas_coeffs}
	\begin{equation} \label{eqn:douglas_coeffs:main}
		\begin{alignedat}{2}
			\A\comp{q,m}&=\A\I=\begin{bmatrix}0 & 0 \\ 1-\theta & \theta\end{bmatrix}, \text{ for } m \leq q,
			&\qquad
			\b\comp{q}&= \b\I = \begin{bmatrix} 1-\theta \\ \theta \end{bmatrix}, \\
			\A\comp{q,m}&=\A\E=\begin{bmatrix}0 & 0 \\ 1 & 0\end{bmatrix}, \text{ for } m > q,
			& \qquad \c\comp{q} &= \c\I = \begin{bmatrix} 0 \\ 1 \end{bmatrix},
		\end{alignedat}
	\end{equation}
	for $q,m = 1, \dots, \nparts$.  The nonstiff partition $f\comp{0}$ prepends the coefficients
	\begin{equation} \label{eqn:douglas_coeffs:0}
		\begin{split}
			\A\comp{0,0} = \c\comp{0} = \begin{bmatrix}
				0
			\end{bmatrix},
			\quad
			\A\comp{q,0} = \begin{bmatrix}
				0 \\ 1
			\end{bmatrix},
			\quad
			\A\comp{0,q} = \begin{bmatrix}
				0 & 0
			\end{bmatrix},
			\quad
			\b\comp{0} = \begin{bmatrix}
				1
			\end{bmatrix},
		\end{split}
	\end{equation}
\end{subequations}
to the tableau, with $q = 1, \dots, \nparts$.  We observe that the method is internally consistent and stiffly accurate.  Considering \cref{eqn:GARK-order-conditions}, we immediately check that \cref{eqn:douglas_coeffs} is a first order method for arbitrary $f\comp{0}$ and all $\theta\in\R$. In fact, the method reaches order two iff $f\comp{0}\equiv 0$ and $\theta=\frac{1}{2}$.  

\ifreport
Furthermore, the Douglas method \cref{eqn:douglas_coeffs} admits the following natural lower triangular formulation
\begin{equation*}
	\begin{butchertableau}{c|c:c:c}
		& f\comp{0}_1 & f\comp{1:\nparts}_1 & f\comp{1:\nparts}_2
		\\ \hline
		Y\comp{0}_1 & 0 & \zero{\nparts}^\tp & \zero{\nparts}^\tp \\\hdashline
		Y\comp{1:\nparts}_1 & \zero{\nparts} & \zero{\nparts}{\nparts} & \zero{\nparts}{\nparts}
		\\\hdashline 
		Y\comp{1:\nparts}_2 & \one{\nparts} & \one{\nparts}{\nparts} -\theta \, \tril{\nparts} & \theta \, \tril{\nparts}
		\\\hline
		y_{n+1} & 1 & (1-\theta) \, \one{\nparts}^\tp & \theta \, \one{\nparts}^\tp
	\end{butchertableau},
%	\quad
%	\tril{\nparts}=\begin{bmatrix}
%		1 & \ldots & 0\\
%		\vdots  & \ddots & \vdots\\
%		1 &  \ldots & 1
%	\end{bmatrix} \in \R{\nparts}{\nparts}.
\end{equation*}
where
\begin{equation}\label{eqn:coefmatrices-triangformN}
	\tril{\nparts}=\begin{bmatrix}
		1 & \ldots & 0\\
		\vdots  & \ddots & \vdots\\
		1 &  \ldots & 1
	\end{bmatrix} \in \R{\nparts}{\nparts}.
\end{equation}
\fi

\begin{remark}
	Modified versions of the Douglas scheme \cref{eqn:splitting-Douglas} have been considered in \cite{Arraras2017,Hundsdorfer_2002_splitting} in such a way that these modifications provide second order approximation when $\theta=\frac{1}{2}$ regardless of the splitting term $f\comp{0}$. A first modified Douglas method is obtained by introducing an initial stabilizing correction for $f\comp{0}$ performed as
	\begin{equation}
		\label{eqn:splitting-modif-Douglas-f0first}
		\begin{split}
			\widehat{v}_0 &= y_n+\Dt \, f(t_n,y_n)
			\\
			v_0 &= \widehat{v}_0+\theta \, \Dt \left( f\comp{0}(t_{n+1},\widehat{v}_0)- f\comp{0}(t_{n},y_n) \right),
			\\
			v_{q} &= v_{q-1}+\theta\,\Dt \left( f\comp{q}(t_{n+1},v_{q})- f\comp{q}(t_{n},y_n) \right),\qquad q = 1,\ldots,\nparts,
			\\
			y_{n+1} &= v_{\nparts}.
		\end{split}
	\end{equation}
	The method \cref{eqn:splitting-modif-Douglas-f0first} has GARK stages
	\comment[id=2]{Fixed the summation indexing}
	\begin{align*}
		Y\comp{q}_1 &= y_n, \qquad q = 0, \dots, \nparts, \\
		Y\comp{0}_2 &= y_n + \Dt \sum_{m=0}^{\nparts} f\comp{m} \mleft(t_{n}, Y\comp{m}_1 \mright), \\
		Y\comp{q}_2 &= y_n + (1 - \theta) \, \Dt \sum_{m=0}^q f\comp{m} \mleft(t_{n}, Y\comp{m}_1 \mright) + \Dt \sum_{m=q+1}^{\nparts} f\comp{m} \mleft(t_{n}, Y\comp{m}_1 \mright) \\
		& \quad + \theta \, \Dt \sum_{m=0}^q f\comp{m} \mleft(t_{n+1}, Y\comp{m}_2 \mright), \qquad q = 1, \dots, \nparts, \\
		y_{n+1} &= Y\comp{\nparts}_2.
	\end{align*}
	The coefficient matrices for partitions $1$ to $\nparts$ are identical to the original Douglas splitting coefficients \cref{eqn:douglas_coeffs:main}.  For partition $f\comp{0}$, however, we replace \cref{eqn:douglas_coeffs:0} with
	\begin{alignat*}{2}
		\A\comp{0,0} &= \A\comp{0,q} = \begin{bmatrix}
			0 & 0 \\ 1 & 0
		\end{bmatrix}, & \qquad
		\A\comp{q,0} &= \begin{bmatrix}
			0 & 0 \\ 1-\theta & \theta
		\end{bmatrix}, \\
		\b\comp{0}* &= \begin{bmatrix}
			1-\theta & \theta
		\end{bmatrix}, & \qquad
		\c\comp{0}* &= \begin{bmatrix}
			0 & 1
		\end{bmatrix},
	\end{alignat*}
	for $q = 1, \dots, \nparts$.  From here, \cref{eqn:splitting-modif-Douglas-f0first} is a first order method for arbitrary $f\comp{0}$ and all $\theta\in\R$. Moreover, the method reaches order two iff $\theta=\frac{1}{2}$, regardless of $f\comp{0}$. 
	\ifreport
	A lower triangular formulation of the modified Douglas method \cref{eqn:splitting-modif-Douglas-f0first} is given by
	\begin{equation*}
		\begin{butchertableau}{c|c:c:c:c}
			& f_1\comp{0} & f\comp{1:\nparts}_1 & f_2\comp{0} & f\comp{1:\nparts}_2
			\\ \hline
			Y_1\comp{0} & 0 & \zero{\nparts}^\tp & 0 & \zero{\nparts}^\tp\\\hdashline
			Y\comp{1:\nparts}_1 & \zero{\nparts} & \zero{\nparts}{\nparts} & \zero{\nparts} & \zero{\nparts}{\nparts}
			\\\hdashline 
			Y_2\comp{0} & 1 & \one{\nparts}^\tp &  0 &  \zero{\nparts}^\tp\\\hdashline
			Y\comp{1:\nparts}_2 & (1-\theta) \, \one{\nparts} & \one{\nparts}{\nparts} -\theta \, \tril{\nparts} & \theta \, \one{\nparts} & \theta \, \tril{\nparts}
			\\\hline
			y_{n+1} & 1-\theta & (1-\theta) \, \one{\nparts}^\tp & \theta & \theta \, \one{\nparts}^\tp
		\end{butchertableau}.
	\end{equation*}
	\fi
	
	On the other hand, a second modification of the  Douglas method \cref{eqn:splitting-Douglas} is obtained by introducing a stabilizing correction for $f\comp{0}$ performed at the end of a step:
	\begin{equation}
		\label{eqn:splitting-modif-Douglas-f0last}
		\begin{aligned}
			v_0&= y_n+\Dt \, f(t_n,y_n)
			\\
			v_{q}&= v_{q-1}+\theta\,\Dt \left( f\comp{q}(t_{n+1},v_{q})- f\comp{q}(t_{n},y_n) \right), \qquad q = 1,\ldots,\nparts,
			\\
			y_{n+1}&= v_{\nparts}+\theta \, \Dt \left( f\comp{0}(t_{n+1},v_{\nparts})- f\comp{0}(t_{n},y_n) \right).
		\end{aligned}
	\end{equation}
	The corresponding stages are
	\comment[id=2]{Fixed the summation indexing}
	\begin{align*}
		Y\comp{q}_1 &= y_n, \qquad q = 0, \dots, \nparts, \\
		Y\comp{q}_2 &= y_n + \Dt \,f\comp{0} \mleft(t_{n}, Y\comp{0}_1 \mright) \\
		&\quad + (1 - \theta) \, \Dt \sum_{m=1}^q f\comp{m} \mleft( t_{n},Y\comp{m}_1 \mright) + \Dt \sum_{m=q+1}^{\nparts} f\comp{m} \mleft(t_{n}, Y\comp{m}_1 \mright) \\
		& \quad + \theta \, \Dt \sum_{m=1}^q f\comp{m} \mleft(t_{n+1},Y\comp{m}_2 \mright),\qquad q = 1, \dots, \nparts, \\
		Y\comp{0}_2 &= y_n + \Dt \, f\comp{0} \mleft( t_{n},Y\comp{0}_1 \mright) \\
		&\quad + \Dt \sum_{m=1}^{\nparts} \left( (1 - \theta) \, f\comp{m} \mleft(t_{n}, Y\comp{m}_1 \mright) + \theta \, f\comp{m} \mleft(t_{n+1}, Y\comp{m}_2 \mright) \right), \\
		y_{n+1} &= y_n + \Dt \sum_{m=0}^{\nparts} \left( (1 - \theta) \, f\comp{m} \mleft(t_{n}, Y\comp{m}_1 \mright) + \theta \, f\comp{m} \mleft(t_{n+1}, Y\comp{m}_2 \mright) \right).
	\end{align*}
	Once again, the coefficient matrices for partitions $1$ to $\nparts$ coincide with \cref{eqn:douglas_coeffs:main}, and partition $f\comp{0}$ has the coefficients
	\begin{alignat*}{2}
		\A\comp{0,0} &= \A\comp{q,0} = \begin{bmatrix}
			0 & 0 \\ 1  & 0
		\end{bmatrix}, 
		& \qquad
		\A\comp{0,q} &= \begin{bmatrix}
			0 & 0 \\ 1 - \theta & \theta
		\end{bmatrix}, \\
		\b\comp{0}* &= \begin{bmatrix}
			1-\theta & \theta
		\end{bmatrix},
		& \qquad
		\c\comp{0}* &= \begin{bmatrix}
			0 & 1
		\end{bmatrix},
	\end{alignat*}
for $q = 1, \dots, \nparts$.  We observe that in this case the method \cref{eqn:splitting-modif-Douglas-f0last} is not stiffly accurate. Anyway, it is a first order method for arbitrary $f\comp{0}$ and all $\theta\in\R$, and it reaches order two iff $\theta=\frac{1}{2}$, regardless of $f\comp{0}$. 
\ifreport
It also admits a natural lower triangular formulation
	\begin{equation*}
		\begin{butchertableau}{c|c:c:c:c}
			& f\comp{0}_1 & f\comp{1:\nparts}_1 & f\comp{1:\nparts}_2 & f\comp{0}_2
			\\ \hline
			Y\comp{0}_1 & 0 & \zero{\nparts}^\tp  & \zero{\nparts}^\tp & 0
			\\\hdashline
			
			Y\comp{1:\nparts}_1 & \zero{\nparts} & \zero{\nparts}{\nparts} & \zero{\nparts}{\nparts} & \zero{\nparts}
			\\\hdashline 
			
			Y\comp{1:\nparts}_2 & \one{\nparts} & \one{\nparts}{\nparts} -\theta \, \tril{\nparts}  & \theta \, \tril{\nparts} & \zero{\nparts}
			\\\hdashline
			
			Y\comp{0}_2 & 1 & (1-\theta) \, \one{\nparts}^\tp  &  \theta \, \one{\nparts}^\tp &  0
			\\\hline
			y_{n+1} & 1-\theta & (1-\theta) \, \one{\nparts}^\tp  &  \theta \, \one{\nparts}^\tp &  \theta
		\end{butchertableau}
	\end{equation*}
	with $\tril{\nparts}$ as in \cref{eqn:coefmatrices-triangformN}.
	\fi
\end{remark}

%%%%%%%%%%%%%%%%%%
\subsection{Modified Craig-Sneyd scheme}
%%%%%%%%%%%%%%%%%%
%

We consider the split system \cref{eqn:split-with-nonstiff} with one non-stiff component $f\comp{0}$. The second order Modified Craig-Sneyd scheme \cite{HouWel2009} reads:
\begin{equation}
	\label{eqn:splitting-Modified-Craig-Sneyd}
	\begin{split}
		v_0 &= y_n+\Dt\, f(t_n,y_n),\\
		v_{q}&= v_{q-1}+\theta \, \Dt \left( f\comp{q}(t_{n+1},v_{q})- f\comp{q}(t_{n},y_n)\right),\qquad q = 1,\ldots,\nparts,\\
		\widehat{v}_0 &= v_0 + \sigma \, \Dt \left( f\comp{0} (t_{n+1}, v_{\nparts}) - f\comp{0} (t_{n}, y_{n}) \right), \\
		w_0&= \widehat{v}_{0}+ \mu \, \Dt \left( f(t_{n+1},v_\nparts)-f(t_n,y_n) \right),\\
		w_{q} &= w_{q-1}+\theta\,\Dt \left( f\comp{q}(t_{n+1},w_{q})- f\comp{q}(t_{n},y_n) \right),\qquad q = 1,\ldots,\nparts,\\
		y_{n+1}&= w_{\nparts}.
	\end{split}
\end{equation}
\Cref{eqn:splitting-Modified-Craig-Sneyd} is order two iff $\sigma = \theta$ and $\mu = \frac{1}{2} - \theta$.  In the special case $\mu = 0$, we recover the so called Craig-Sneyd scheme \cite{Craig_1988_ADI}.

The method is cast in the GARK form \cref{eqn:consistent-Butcher-tableau}  by defining the stages 
\begin{equation} \label{eqn:Modified-Craig-Sneyd-stages}
	Y\comp{0} \coloneqq \begin{bmatrix}
		y_n \\ v_\nparts
	\end{bmatrix}, \qquad
	Y\comp{q} \coloneqq \begin{bmatrix}
		y_n\\v_{q}\\v_{\nparts}\\w_{q}
	\end{bmatrix}, \text{ for } q = 1, \dots, \nparts.
\end{equation}
The coefficients correspond to that of an ADI-GARK method:
\begin{align*}
	\A\comp{q,m}&=\A\I=\begin{bmatrix}
		0 & 0 &0 &0\\
		1-\theta & \theta & 0 & 0\\
		1-\theta & \theta & 0 & 0\\
		1-\mu-\theta & 0 & \mu & \theta
	\end{bmatrix}, \text{ for } m \leq q, \\
	\A\comp{q,m}&=\A\E = \begin{bmatrix}
		0 & 0 & 0 & 0\\
		1 & 0 & 0 & 0\\
		1-\theta & \theta & 0 & 0\\
		1-\mu & 0 & \mu & 0
	\end{bmatrix}, \text{ for } m > q, \\
	\b\comp{q}* &= \b\I*=\begin{bmatrix}1-\mu-\theta & 0 & \mu & \theta \end{bmatrix},\qquad 
	\c\comp{q}* = \c\I* = \begin{bmatrix} 0 & 1 & 1  & 1  \end{bmatrix}.
\end{align*}
The nonstiff partition introduces the following additional coefficients:
\begin{alignat*}{4}
	\A\comp{0,0} &= \begin{bmatrix}
		0 & 0 \\ 1 & 0
	\end{bmatrix}, & \quad
	\A\comp{0,q} &= \begin{bmatrix}
		0 & 0 & 0 & 0 \\
		1 - \theta & \theta & 0 & 0
	\end{bmatrix}, \\
	\A\comp{q,0} &= \begin{bmatrix}
		0 & 0 \\
		1 & 0 \\
		1 & 0 \\
		1 - \sigma - \mu & \sigma + \mu
	\end{bmatrix}, & \quad
	\b\comp{0} &= \begin{bmatrix}
		1 - \sigma - \mu \\  \sigma + \mu
	\end{bmatrix}, & \quad
		\c\comp{0} &= \begin{bmatrix}
		0 \\ 1
	\end{bmatrix}.
\end{alignat*}
\ifreport
Alternatively, the method can be expressed in lower triangular form
\begin{equation*}
	\resizebox{\textwidth}{!}{$
	\begin{butchertableau}{c|c:c:c:c:c:c}
		& f\comp{0}_1 & f\comp{1:\nparts}_1 & f\comp{1:\nparts}_2 & f\comp{0}_2 & f\comp{1:\nparts}_3 & f\comp{1:\nparts}_4 \\ \hline
		Y\comp{0}_1 & 0 & \zero{\nparts}^\tp & \zero{\nparts}^\tp & 0 & \zero{\nparts}^\tp & \zero{\nparts}^\tp \\ \hdashline
		Y\comp{1:\nparts}_1 & \zero{\nparts} & \zero{\nparts}{\nparts} & \zero{\nparts}{\nparts} & \zero{\nparts} & \zero{\nparts}{\nparts} & \zero{\nparts}{\nparts} \\ \hdashline 
		Y\comp{1:\nparts}_2 & \one{\nparts} & \one{\nparts}{\nparts} - \theta \, \tril{\nparts} & \theta \, \tril{\nparts} & \zero{\nparts} & \zero{\nparts}{\nparts} & \zero{\nparts}{\nparts} \\ \hdashline
		Y\comp{0}_2 & 1 & (1 - \theta) \, \one{\nparts}^{\tp} & \theta \, \one{\nparts}^{\tp} & 0 & \zero{\nparts}^\tp & \zero{\nparts}^\tp \\ \hdashline
		Y\comp{1:\nparts}_3 & \one{\nparts} & (1 - \theta) \, \one{\nparts}{\nparts} & \theta \, \one{\nparts}{\nparts} & \zero{\nparts} & \zero{\nparts}{\nparts} & \zero{\nparts}{\nparts} \\ \hdashline
		Y\comp{1:\nparts}_4 & (1 - \sigma - \mu) \, \one{\nparts} & (1 - \mu) \, \one{\nparts}{\nparts} - \theta \, \tril{\nparts} & \zero{\nparts}{\nparts} & (\sigma + \mu) \, \one{\nparts} & \mu \, \one{\nparts}{\nparts} & \theta \, \tril{\nparts} \\ \hline
		y_{n+1} & 1 - \sigma - \mu & (1 - \mu - \theta) \one{\nparts}^\tp & \zero{\nparts}^\tp & \sigma + \mu & \mu \one{\nparts}^\tp & \theta \one{\nparts}^\tp
	\end{butchertableau}.
	$}
\end{equation*}
\fi
%
%%%%%%%%%%%%%%%%%%
\subsection{Hundsdorfer--Verwer scheme}
%%%%%%%%%%%%%%%%%%
%
We consider the split system \cref{eqn:split-with-nonstiff} with one non-stiff component $f\comp{0}$.  The Hundsdorfer--Verwer scheme \cite{HouWel2009} reads:
\begin{equation}
	\label{eqn:splitting-Hundsdorfer-Verwer}
	\begin{split}
		v_0 &= y_n+\Dt\, f(t_n,y_n),\\
		v_{q} &= v_{q-1}+\theta \,\Dt\,\left( f\comp{q}(t_{n+1},v_{q})- f\comp{q}(t_{n},y_n)\right),\qquad q = 1, \ldots, \nparts,\\
		w_0&= v_{0} + \mu \, \Dt\,\Bigl(f(t_{n+1},v_\nparts)-f(t_n,y_n)\Bigr),\\
		w_{q} &= w_{q-1}+\theta \,\Dt\,\left( f\comp{q}(t_{n+1},w_{q})- f\comp{q}(t_{n+1},v_\nparts)\right),\qquad q = 1, \ldots, \nparts,\\
		y_{n+1}&= w_{\nparts}.
	\end{split}
\end{equation}
The method is cast in the GARK form \cref{eqn:consistent-Butcher-tableau} using the same stages as \cref{eqn:Modified-Craig-Sneyd-stages} and the coefficient matrices
\begin{align*}
	\A\comp{q,m}&=\A\I=\begin{bmatrix}0 & 0 &0 &0\\ 1-\theta & \theta & 0 & 0\\ 1-\theta & \theta & 0 & 0\\
	1-\mu & 0 & \mu-\theta & \theta
	\end{bmatrix}, \text{ for } m \leq q; \\
	\A\comp{q,m}&=\A\E = \begin{bmatrix}0 & 0 & 0 & 0\\1 & 0 & 0 & 0\\ 1-\theta & \theta & 0 & 0\\
	1-\mu & 0 & \mu & 0
	\end{bmatrix}, \text{ for } m > q, \\
	\b\comp{q}* &= \b\I*=\begin{bmatrix} 1-\mu & 0 & \mu-\theta & \theta \end{bmatrix}, \qquad
	\c\comp{q}* = \c\I* = \begin{bmatrix} 0 & 1 & 1  & 1  \end{bmatrix},
\end{align*}
together with
\begin{alignat*}{4}
	\A\comp{0,0} &= \begin{bmatrix}
		0 & 0 \\ 1 & 0
	\end{bmatrix}, & \quad
	\A\comp{0,q} &= \begin{bmatrix}
		0 & 0 & 0 & 0 \\
		1 - \theta & \theta & 0 & 0
	\end{bmatrix}, & \quad
	\A\comp{q,0} &= \begin{bmatrix}
		0 & 0 \\
		1 & 0 \\
		1 & 0 \\
		1 - \mu & \mu
	\end{bmatrix}, \\
	\b\comp{0} &= \begin{bmatrix}
		1-\mu \\ \mu
	\end{bmatrix}, & \quad
	\c\comp{0} &= \begin{bmatrix}
		0 \\ 1
	\end{bmatrix}.
\end{alignat*}
The Hundsdorfer--Verwer Splitting method \cref{eqn:splitting-Hundsdorfer-Verwer} is an ADI-GARK scheme, is stiffly accurate, and has order two iff $\mu = \frac{1}{2}$.  
\ifreport
Alternatively, the method can be expressed in lower triangular form
\begin{equation*}
	\resizebox{\textwidth}{!}{$
	\begin{butchertableau}{c|c:c:c:c:c:c}
		& f\comp{0}_1 & f\comp{1:\nparts}_1 & f\comp{1:\nparts}_2 & f\comp{0}_2 & f\comp{1:\nparts}_3 & f\comp{1:\nparts}_4 \\ \hline
		Y\comp{0}_1 & 0 & \zero{\nparts}^\tp & \zero{\nparts}^\tp & 0 & \zero{\nparts}^\tp & \zero{\nparts}^\tp \\ \hdashline
		Y\comp{1:\nparts}_1 & \zero{\nparts} & \zero{\nparts}{\nparts} & \zero{\nparts}{\nparts} & \zero{\nparts} & \zero{\nparts}{\nparts} & \zero{\nparts}{\nparts} \\ \hdashline 
		Y\comp{1:\nparts}_2 & \one{\nparts} & \one{\nparts}{\nparts} - \theta \, \tril{\nparts} & \theta \, \tril{\nparts} & \zero{\nparts} & \zero{\nparts}{\nparts} & \zero{\nparts}{\nparts} \\ \hdashline
		Y\comp{0}_2 & 1 & (1 - \theta) \, \one{\nparts}^{\tp} & \theta \, \one{\nparts}^{\tp} & 0 & \zero{\nparts}^\tp & \zero{\nparts}^\tp \\ \hdashline
		Y\comp{1:\nparts}_3 & \one{\nparts} & (1 - \theta) \, \one{\nparts}{\nparts} & \theta \, \one{\nparts}{\nparts} & \zero{\nparts} & \zero{\nparts}{\nparts} & \zero{\nparts}{\nparts} \\ \hdashline
		Y\comp{1:\nparts}_4 & (1 - \mu) \, \one{\nparts} & (1 - \mu) \, \one{\nparts}{\nparts} & \zero{\nparts}{\nparts} & \mu \, \one{\nparts} & \mu \, \one{\nparts}{\nparts} - \theta \, \tril{\nparts} & \theta \, \tril{\nparts} \\ \hline
		y_{n+1} & 1 - \mu & (1 - \mu) \, \one{\nparts}^\tp & \zero{\nparts}^\tp & \mu & (\mu - \theta) \, \one{\nparts}^\tp & \theta \, \one{\nparts}^\tp
	\end{butchertableau}.
	$}
\end{equation*}
\fi
%
	%!TEX root = ../main.tex
%%%%%%%%%%%%%%%%%%%%%%%%%%%%%%%%%%%%%%%%%%%%%%%%%%%%%%%%%%%%%%%
\section{Classical operator splitting methods in the GARK framework}
\label{sec:operator-splitting-methods}
%%%%%%%%%%%%%%%%%%%%%%%%%%%%%%%%%%%%%%%%%%%%%%%%%%%%%%%%%%%%%%%

We next describe how classical operator splitting schemes appearing in the literature can be formulated within the GARK framework.

%%%%%%%%%%%%%%%%%%
\subsection{The second order Strang splitting scheme}
%%%%%%%%%%%%%%%%%%
%
A Strang splitting \cite{Strang_1968_splitting} integrates individual sub-systems in sequence, a symmetrical order. 
For an arbitrary number $\nparts$ of partitions, the Strang split formulation \cite{Strang_1968_splitting} is
\begin{alignat*}{4}
	v_0 &= y_n \\
	v_q' &= f\comp{q}(t,v_{q}), & \quad v_q(t_n) &= v_{q-1}\mleft(t_n + \tfrac{\Dt}{2} \mright), & \quad & t_n \le t \le t_n + \tfrac{\Dt}{2}, \\
	& \qquad q = 1,\dots,\nparts, \\
	w_0 &= v_{\nparts}, \\
	w_q'&= f\comp{\nparts-q+1}(t,w_{q}), & \quad w_q \mleft(t_n + \tfrac{\Dt}{2} \mright) &= w_{q-1}(t_n + \Dt), & \quad & t_n+ \tfrac{\Dt}{2} \le t \le t_n + \Dt, \\
	&\qquad q = 1,\dots,\nparts, \\
	y_{n+1} &= w_\nparts(t_n+\Dt).
\end{alignat*}
Assume that each integration is carried out by one step of an arbitrary, $s$ stage Runge--Kutta scheme $(A,b,c)$.
The method is cast in the GARK form \cref{eqn:consistent-Butcher-tableau} by defining the stages 
\comment{Fixed the definition to include all $2s$ stages per partition}
\begin{equation*}
	Y\comp{q}=\begin{bmatrix} Y\comp{q}_{1:s}~(\textnormal{from integration of}~v_q) \\ Y\comp{q}_{s+1 : 2s}~(\textnormal{from integration of}~w_{\nparts-q+1}) \end{bmatrix}.
\end{equation*}
The Butcher tableau is of the form \cref{eqn:Butcher-tableau-structure-2} with the coefficient matrices
\begin{alignat*}{1}
	\A\comp{q,m}&=\A\L=\begin{smatrix}\frac{1}{2}\, \one \, b^\tp & \zero \\ \frac{1}{2}\, \one \, b^\tp  & \zero \end{smatrix}, \text{ for } m < q,
	 \\
	\A\comp{q,q}&=\A\D=\begin{smatrix}\frac{1}{2}\, A & \zero \\ \frac{1}{2}\, \one \, b^\tp  & \frac{1}{2}\,A \end{smatrix}, 
	\qquad
	\b\comp{q}= \b\D = \begin{smatrix} \frac{1}{2} b \\ \frac{1}{2} b \end{smatrix}, \\
	\A\comp{q,m}&=\A\U=\begin{smatrix} 0 & 0 \\\frac{1}{2}\, \one \, b^\tp&\frac{1}{2}\, \one \, b^\tp\end{smatrix}, \text{ for } m > q.
%	& \qquad \c\comp{q} &= \c = \begin{smatrix} 0 \\ 1 \end{smatrix}.
\end{alignat*}

\ifreport
The corresponding permuted matrix of coefficients \cref{eqn:permuted-Butcher-tableau}, with $m < q < r$, is
\begin{equation*}
	\resizebox{\textwidth}{!}{$
	\begin{butchertableau}{c|c|c:c|c|c:c|c|c:c|c}
		&  \dots & f\comp{m}_{1:s} & f\comp{m}_{s+1:2s} & \dots &  f\comp{q}_{1:s} & f\comp{q}_{s+1:2s} &  \dots &f\comp{r}_{1:s} & f\comp{r}_{s+1:2s} & \dots   \\
		\hline \hdashline \hline
		Y\comp{q}_{1:s} & \dots &  \frac{1}{2}\, \one \, b^\tp& \zero & \dots &\frac{1}{2} \, A & \zero & \dots &\zero  & \zero & \dots \\
		\hdashline
		Y\comp{q}_{s+1:2s} & \dots & \frac{1}{2}\, \one\,  b^\tp & \zero  & \dots & \frac{1}{2}\, \one \, b^\tp& \frac{1}{2} \, A & \dots & \frac{1}{2}\, \one \, b^\tp   & \frac{1}{2}\, \one \, b^\tp & \dots \\
		\hline\hdashline\hline
		y_{n+1} & \dots & \frac{1}{2} \, b^\tp  & \frac{1}{2} \, b^\tp & \dots & \frac{1}{2} \, b^\tp & \frac{1}{2} \, b^\tp & \dots & \frac{1}{2} \, b^\tp & \frac{1}{2} \, b^\tp & \dots
	\end{butchertableau}.
	$}
\end{equation*}
\fi

%%%%%%%%%%%%%%%%%%
\subsection{High order Yoshida splitting schemes}
%%%%%%%%%%%%%%%%%%

%Splitting schemes of order higher than two require backward-in-time integration of each of the component subsystems.
Yoshida proposed high order splitting schemes \cite{Yoshida_1990_splitting} that require backward-in-time integration. Let $y_{n+1} = S_{\Dt}\, y_n$ denote formally the Strang splitting solution operator. Then Yoshida's fourth order splitting is obtained by a repeated application of Strang solutions:
\begin{equation} \label{eqn:fourth_order_yoshida}
	y_{n+1} = S_{\theta\Dt}\,S_{(1-2\theta)\Dt}\,S_{\theta\Dt}\, y_n, \quad \theta = \frac{1}{2 - 2^{1/3}} \approx 1.35121.
\end{equation}
General Yoshida schemes applied to a two-way partitioned system read
\comment[id=2]{Removed definition of Yoshida splitting with undefined $F^{\{1\}}$ and $F^{\{2\}}$ notation}
%
% \begin{equation} \label{eqn:yoshida}
% 	y_{n+1} = e^{\Dt \beta_L F\comp{2}}\,e^{\Dt \alpha_L F\comp{1}} \dots e^{\Dt \beta_1 F\comp{2}}\,e^{\Dt \alpha_1 F\comp{1}} y_n.
% \end{equation}
%
%
\begin{align*}
	v_\ell(t_n + c_{\ell-1}^{\alpha}\,\Dt) &= \begin{cases}
	y_n & \ell = 1 \\
	w_{\ell-1}(t_n + c_{\ell-1}^{\beta}\,\Dt), & \text{otherwise}
	\end{cases} \\
	v_\ell'&= f\comp{1}(v_{\ell}), \qquad t_n + c_{\ell-1}^{\alpha}\,\Dt  \le t \le t_n + c_{\ell}^{\alpha}\,\Dt, \\
	w_\ell(t_n + c_{\ell-1}^{\beta}\,\Dt) &= v_\ell(t_n + c_{\ell}^{\alpha}\,\Dt), \\
	w_\ell'&= f\comp{2}(w_{\ell}), \qquad  t_n + c_{\ell-1}^{\beta}\,\Dt  \le t \le t_n + c_{\ell}^{\beta}\,\Dt, \\
	& \qquad \ell = 1, \dots, L \\
	y_{n+1} &= w_L(t_n+h),
\end{align*}
where
\begin{equation*}
	c_0^{\alpha} = 0, \quad c_0^{\beta} = 0, \quad
	c_\ell^{\alpha} = \sum_{i=1}^{\ell} \alpha_i, \quad c_\ell^{\beta} = \sum_{i=1}^{\ell} \beta_i, \quad
	c_L^{\alpha} = 1, \quad c_L^{\beta} = 1.
\end{equation*}
The scheme of order four \cref{eqn:fourth_order_yoshida} uses $L=4$ and the coefficients
\begin{alignat*}{3}
	\alpha_1 &= \alpha_4 = \frac{1}{2 \,( 2 - 2^{1/3} )}, & \quad
	\alpha_2 &= \alpha_3 = \frac{1-2^{1/3}}{2 \,( 2 - 2^{1/3} )}, \\
	\beta_1 &= \beta_3 = \frac{1}{2 - 2^{1/3}}, & \quad
	\beta_2 &= -\frac{2^{1/3}}{2 - 2^{1/3}}, & \quad
	\beta_4 &= 0.
\end{alignat*}

Assume that each integration is carried out with a $s$ stage Runge--Kutta scheme $(A,b,c)$.  \added[id=2]{Then, we can define the GARK stages}
\comment[id=2]{Define the GARK stages for Yoshida splitting.}
\begin{align*}
	Y\comp{1} &= \begin{bmatrix} Y\comp{1}_{1:s}~(\textnormal{from integration of}~v_1) \\ \vdots \\ Y\comp{1}_{(L-1) s + 1 : L s}~(\textnormal{from integration of}~v_{L}) \end{bmatrix}, \\
	Y\comp{2} &= \begin{bmatrix} Y\comp{2}_{1:s}~(\textnormal{from integration of}~w_1) \\ \vdots \\ Y\comp{2}_{(L-1) s + 1 : L s}~(\textnormal{from integration of}~w_{L}) \end{bmatrix}.
\end{align*}
The GARK tableau \cref{eqn:consistent-Butcher-tableau} for \cref{eqn:fourth_order_yoshida} reads
\begin{equation*}
	\resizebox{\textwidth}{!}{$
	\begin{butchertableau}{c|c:c:c:c|c:c:c:c}
		c_0^{\alpha} \, \one + \alpha_1 \, c & \alpha_1 \, A &  &   &  &  &  &  &  \\
		\hdashline
		c_1^{\alpha} \, \one + \alpha_2 \, c  & \alpha_1 \, \one \, b^\tp  &  \alpha_2 \, A  &  &  & \beta_1\, \one \, b^\tp&  &  &  \\
		\hdashline
		c_2^{\alpha} \, \one + \alpha_3 \, c & \alpha_1 \, \one \, b^\tp  &  \alpha_2\, \one \, b^\tp   &  \alpha_3 \, A & & \beta_1\, \one \, b^\tp& \beta_2\, \one \, b^\tp &  &  \\
		\hdashline
		c_3^{\alpha} \, \one + \alpha_4 \, c  & \alpha_1 \, \one \, b^\tp  &  \alpha_2\, \one \, b^\tp   &  \alpha_3\, \one \, b^\tp & \alpha_4 \, A  & \beta_1\, \one \, b^\tp & \beta_2\, \one \, b^\tp &  \beta_3\, \one \, b^\tp &  \\
		\hline
		c_0^{\beta} \, \one + \beta_1 \, c & \alpha_1 \, \one \, b^\tp  &   &  &  & \beta_1 \, A&  &  &  \\
		\hdashline
		c_1^{\beta} \, \one + \beta_2 \, c & \alpha_1 \, \one \, b^\tp  &  \alpha_2\, \one \, b^\tp   &  &  & \beta_1\, \one \, b^\tp& \beta_2 \, A&  &  \\
		\hdashline
		c_2^{\beta} \, \one + \beta_3 \, c & \alpha_1 \, \one \, b^\tp  &  \alpha_2\, \one \, b^\tp   &  \alpha_3\, \one \, b^\tp && \beta_1\, \one\,  b^\tp & \beta_2\, \one \, b^\tp &  \beta_3 \, A &  \\
		\hdashline
		c_3^{\beta} \, \one + \beta_4 \, c & \alpha_1 \, \one \, b^\tp  &  \alpha_2\, \one \, b^\tp   &  \alpha_3\, \one \, b^\tp & \alpha_4 \, \one \, b^\tp  & \beta_1\, \one \, b^\tp & \beta_2\, \one \, b^\tp &  \beta_3\, \one \, b^\tp &  \beta_4 \, A  \\
		\hline
		&\alpha_1 \, b^\tp  &  \alpha_2\, b^\tp   &  \alpha_3 \, b^\tp & \alpha_4 \, b^\tp & \beta_1\, b^\tp & \beta_2 \, b^\tp&  \beta_3 \, b^\tp &  \beta_4 \, b^\tp
	\end{butchertableau}.
	$}
\end{equation*}

	%!TEX root = ../main.tex
%%%%%%%%%%%%%%%%%%%%%%%%
\section{Fractional step Runge--Kutta methods in the GARK framework}
\label{sec:FSRK}
%%%%%%%%%%%%%%%%%%%%%%%%

Consider the additively split ODE \cref{eqn:split-with-nonstiff} where the first $\nparts$ components are one-dimensional split linear operators, and the last partition is a nonlinear term:
\begin{equation}
	\label{eqn:additive-ode-ADI}
	 y'= f(t,y) = \sum_{m=1}^{\nparts} \left( \mathbf{L}\comp{m}(t)\,y + \phi\comp{m}(t)\right) + g(t,y).
\end{equation}
Specifically, $\mathbf{L}\comp{m}(t)$ represents the diffusion operator in the $m$ spatial direction, $\phi\comp{m}(t)$ the corresponding directional boundary conditions and source terms, and $g(t,y)$ is a nonlinear term.

Fractional step Runge--Kutta (FSRK) methods \cite{Bujanda_2003_FSRK,Bujanda_2006_FSRK,Bujanda_2004_FSRK,Bujanda_2007_FSRK} solve \cref{eqn:additive-ode-ADI} in the alternating implicit fashion
%\textcolor{purple}{What is $T_j$ below?, where is it defined????}
%
\begin{align*}
	Y_i &= y_n + \Dt \sum_{j=1}^{i} a\comp{m_j}_{i,j} \left( \mathbf{L}\comp{m_j}(T_j)\,Y_j + \phi\comp{m_j}(T_j)\right) + \Dt \sum_{j=1}^{i-1} a\comp{0}_{i,j} \, g(T_j,Y_j), \\
	y_{n+1} &= y_n + \Dt \sum_{j=1}^s b\comp{m_j}_{j} \left( \mathbf{L}\comp{m_j}(T_j)\,Y_j + \phi\comp{m_j}(T_j)\right) + \Dt \sum_{j=1}^s b\comp{0}_{j} \, g(T_j,Y_j),
\end{align*}
where $T_j = t_n + c\comp{m_j}_j$.
We note that only one dimension $m_j$ is associated with each stage $j$, and that each stage solution requires solving a linear system corresponding to a one-dimensional problem. 
\ifreport
This linear system has the form
\begin{align*}
	&\left(\eye{\nvar} - \Dt \, a\comp{m_i}_{i,i} \, \mathbf{L}\comp{m_i}(T_i) \right)\,Y_i = y_n + \Dt\,a\comp{m_i}_{i,i} \, \phi\comp{m_i}(T_i) \\
	& \qquad + \Dt \sum_{j=1}^{i-1} a\comp{m_j}_{i,j} \left( \mathbf{L}\comp{m_j}(T_j)\,Y_j + \phi\comp{m_j}(T_j)\right) + \Dt \sum_{j=1}^{i-1} a\comp{0}_{i,j} \, g(T_j,Y_j).
\end{align*}
\fi
As stated in \cite{Bujanda_2003_FSRK}, FSRK are a special case of the additive Runge--Kutta (ARK) method.
\ifreport
Specifically,
\begin{align*}
	Y_i &= y_n + \Dt \sum_{m=1}^{\nparts} \sum_{j=1}^{i} a\comp{m}_{i,j} \left( \mathbf{L}\comp{m}(T_j)\,Y_j + \phi\comp{m}(T_j) \right) + \Dt \sum_{j=1}^{i-1} a\comp{0}_{i,j} \, g(T_j,Y_j), \\
	y_{n+1} &= y_n + \Dt \sum_{m=1}^{\nparts} \sum_{j=1}^s b\comp{m}_{j} \left( \mathbf{L}\comp{m}(T_j)\,Y_j + \phi\comp{m}(T_j)\right) + \Dt \sum_{j=1}^s b\comp{0}_{j} \, g(T_j,Y_j),
\end{align*}
where $a\comp{m}_{i,j} \ne 0$ for at most one of the dimensions $m_j$.  The ARK formulation, however, requires awkward padding of the tableau with zeros in order to have in implicitness in one partition at a time.  
\fi
The GARK formulation provides a simpler and more compact representation.  First, let us define the index sets
\begin{equation} \label{eqn:index_sets}
	S_{q} = \left\{ \ell \in \{1, \dots, s\} \mid m_\ell = q \right\} = \left\{ \ell_{q,1}, \dots, \ell_{q,s\comp{q}} \right\},
	\qquad
	q = 1, \dots, \nparts.
\end{equation}
This provides a way to map stage $i \in \{1, \dots, s\}$ of a FSRK method to a stage $j \in \{1, \dots, s\comp{q}\}$ in partition $q$ of a GARK method.  Now the GARK coefficients can be expressed as
\begin{equation} \label{eqn:FSRK_to_GARK}
	\A\comp{q,m} = \left[ a\comp{m}_{\ell_{q,i}, \ell_{m,j}} \right]_{i=1,\dots,s\comp{q}}^{j=1,\dots,s\comp{m}},
	\quad
	\b\comp{q} = \left[ b\comp{q}_{\ell_{q,i}} \right]_{i=1,\dots,s\comp{q}},
	\quad
	q, m = 1, \dots, \nparts,
\end{equation}
with partition zero using the coefficients
\begin{alignat*}{2}
	\A\comp{0,0} &= \left[ a\comp{0}_{i,j} \right]_{i=1,\dots,s}^{j=1,\dots,s},
	& \qquad
	\A\comp{0,q} &= \left[ a\comp{0}_{i,\ell_{q,j}} \right]_{i=1,\dots,s}^{j=1,\dots,s\comp{q}}, \\
	\A\comp{q,0} &= \left[ a\comp{0}_{\ell_{q,i},j} \right]_{i=1,\dots,s\comp{q}}^{j=1,\dots,s},
	& \qquad
	\b\comp{0} &= \left[ b\comp{0}_{i} \right]_{i=1,\dots,s},\quad q = 1, \dots, \nparts.
\end{alignat*}

\begin{example} \comment[id=2]{A new example to explain converting FSRK to GARK}
    Consider the second order FSRK used in \cite[Section 5.1]{portero2004avoiding} with coefficients
    \begin{alignat*}{2}
        \A\comp{1} &= \begin{bmatrix}
            0 & 0 & 0 \\
            0 & \frac{1}{2} & 0 \\
            0 & 1 & 0
        \end{bmatrix}, & \qquad
        \A\comp{2} &= \begin{bmatrix}
            0 & 0 & 0 \\
            \frac{1}{2} & 0 & 0 \\
            \frac{1}{2} & 0 & \frac{1}{2}
        \end{bmatrix}, \\
        \b\comp{1}* &= \begin{bmatrix}
            0 & 1 & 0
        \end{bmatrix}, & \qquad
        \b\comp{2}* &= \begin{bmatrix}
            \frac{1}{2} & 0 & \frac{1}{2}
        \end{bmatrix}.
    \end{alignat*}
    This method is for system with $N = 2$ partitions, and we will ignore partition zero, i.e., $g(t, y) = 0$.  There are $s = 3$ stages which map to partition $1$ or $2$ by
    \begin{equation*}
        m_j = \begin{cases}
            1, & j = 2, \\
            2, & j = 1 \text{ or } 3.
        \end{cases}
    \end{equation*}

    To cast this scheme into the GARK framework, we compute the index sets \cref{eqn:index_sets}:
    \begin{equation*}
        S_1 = \{ \ell_{1,1} \} = \{ 2 \},
        \qquad
        S_2 = \{ \ell_{2,1}, \ell_{2,2} \} = \{ 1, 3 \}.
    \end{equation*}
    Thus, $s\comp{1} = | S_1 | = 1$ and $s\comp{2} = | S_2 | = 2$.  By \cref{eqn:FSRK_to_GARK}, the GARK tableau is given by
    \begin{equation*}
        \begin{butchertableau}{c|cc}
            \frac{1}{2} & \frac{1}{2} & 0 \\ \hline
            0 & 0 & 0 \\
            1 & \frac{1}{2} & \frac{1}{2} \\ \hline
            1 & \frac{1}{2} & \frac{1}{2}
        \end{butchertableau}
    \end{equation*}
    Note that the FSRK formulation requires columns of zeros in $\A\comp{1}$ and $\A\comp{2}$, whereas the GARK formulation is more compact.  By reading the diagonal blocks of the GARK tableau, we can immediately see this method couples the implicit midpoint method with the implicit trapezoidal method.
\end{example}
	%!TEX root = ../main.tex
%%%%%%%%%%%%%%%%%%%%%%%%%%%%%
\section{High-order ADI-GARK methods}
\label{sec:new_methods}
%%%%%%%%%%%%%%%%%%%%%%%%%%%%%

In this section we develop two new ADI-GARK methods (\Cref{definition:ADI-GARk}) of order three and four. The GARK order conditions \cref{eqn:GARK-order-conditions} particularized to an IMIM-GARK method with structure \cref{eqn:Butcher-tableau-structure-2} lead to the following result.

\begin{theorem}[Order conditions for the special class of IMIM-GARK]
    \label{thrm:ADI-GARK-OC}
	Consider an IMIM-GARK method of type \cref{eqn:Butcher-tableau-structure-2} that satisfies the internal consistency condition \cref{eqn:internal-consistency}. Then we have the following:
	\begin{itemize}
		\item The method has order $p \in \{1,2,3\}$ iff each component Runge--Kutta method $(\A\D,\b\D,\c\D)$, $(\A\L,\b\D,\c\D)$, and $(\A\U,\b\D,\c\D)$ has order at least $p$.
		\item The method has order $p=4$ iff each component Runge--Kutta method  $(\A\D,\b)$, $(\A\L,\b\D)$, and $(\A\U,\b\D)$ has order at least $4$, and, in addition, the following coupling order conditions are satisfied:
		\begin{subequations}
			\label{eqn:GARK-coupling-4}
			\begin{alignat}{2}
				\label{eqn:GARK-coupling-4a}
				\b\D*  \A\D   \A\U   \c\D &=  \frac{1}{24},  \qquad & \b\D*   \A\U   \A\D   \c\D &=  \frac{1}{24}, \\
				\label{eqn:GARK-coupling-4b}
				\b\D*   \A\L   \A\U   \c\D& =  \frac{1}{24},  \qquad & \b\D*   \A\U   \A\L   \c\D &=  \frac{1}{24}, \\
				\label{eqn:GARK-coupling-4c}
				\b\D*   \A\D   \A\L   \c\D &=  \frac{1}{24},  \qquad & \b\D*   \A\L   \A\D   \c\D &=  \frac{1}{24}.
			\end{alignat}
		\end{subequations}
	\end{itemize}
\end{theorem}
\begin{proof}
	We set $\b\comp{\sigma} = \b\D $ and \added[id=2]{$\A\comp{\sigma,\mu} \one{s\comp{\mu}} = \c\D$} for all $\sigma,\mu \in \{  \text{L},\text{D}, \text{U} \}$ in \cref{eqn:GARK-order-conditions}. If the  base methods satisfy their  own order conditions, no third order coupling conditions  remain.  At order four, the only remaining coupling conditions are listed in \cref{eqn:GARK-coupling-4}. \qed 
\end{proof}

\begin{remark}
The order conditions for ADI-GARK methods of \Cref{definition:ADI-GARk}, as well as those for parallel ADI-GARK methods of \Cref {definition:P-ADI-GARk}, reduce to \cref{eqn:GARK-coupling-4a}, since \cref{eqn:GARK-coupling-4b} and \cref{eqn:GARK-coupling-4c} are redundant. 
\end{remark}

\subsection{\added[id=1]{(Parallel)} ADI-GARK method of order 3}
\added[id=2]{For the derivation of a third order ADI-GARK method, we start by selecting an implicit scheme.  We use the optimal, L-stable, 4 stage ESDIRK method described in \cite[Section 5.1.1]{kennedy2016diagonally}.  For the explicit counterpart, it must use the $\b\I$ and $\c\I$ coefficients, leaving $\A\E$ as free parameters.  \Cref{thrm:ADI-GARK-OC} requires the explicit method to satisfy classical order conditions up to order three, and we also impose the simplifying assumption $D(1)$ \cite[page 208]{hairer1993solving}.  The one remaining parameter in $\A\E$ is determined by}
\begin{equation*}
    \b\I* \A\E \A\E \c\I = \frac{5}{268}.
\end{equation*}
\added[id=2]{This ensures the linear stability function satisfies $|R( [z,z] )| \leq 1$ for all $z$ in the left-half plane.  This stability region is plotted in \cref{fig:stab-Order3}, and the method coefficients are listed in \cref{tab:ADI-GARK-3}.}  \added[id=1]{By \cref{rem:serial_parallel}, the coefficients also define a third order parallel ADI-GARK method.}
\begin{table}[h]
\centering
\begin{align*}
	\A\I &= \begin{bmatrix}
		0 & 0 & 0 & 0 \\
		\gamma  & \gamma  & 0 & 0 \\
		\frac{215 \gamma +424}{2624-1536 \gamma } & \frac{264-841 \gamma }{1536 \gamma +448} & \gamma  & 0 \\
		\frac{2 \gamma +1}{4 \gamma +8} & \frac{31-14 \gamma }{352-900 \gamma } & \frac{320 \gamma +224}{575-477 \gamma } & \gamma  \\
	\end{bmatrix}, \\
	\A\E &= \begin{bmatrix}
		0 & 0 & 0 & 0 \\
		2 \gamma  & 0 & 0 & 0 \\
		\frac{12526987 \gamma +655304}{8876160 \gamma +7175968} & \frac{15 (215 \gamma +152)}{2144 (92 \gamma -9)} & 0 & 0 \\
		\frac{2370311 \gamma -563481}{134 (17071 \gamma +921)} & \frac{380783-137789 \gamma }{134 (17727 \gamma -15511)} & \frac{1000-304 \gamma }{1371 \gamma +379} & 0 \\
	\end{bmatrix}, \\
	\b\I* &= \begin{bmatrix}
		\frac{2 \gamma +1}{4 \gamma +8} &
		\frac{31-14 \gamma }{352-900 \gamma } &
		\frac{320 \gamma +224}{575-477 \gamma } &
		\gamma
	\end{bmatrix}, \\
	\c\I* &= \begin{bmatrix}
		0 & 2 \gamma & \frac{\gamma +2}{4} & 1 
	\end{bmatrix}.
%	\gamma &\approx 0.43586652150845900 \text{ is the middle root of } 0 = -1 + 9 \gamma - 18 \gamma^2 + 6 \gamma^3.
\end{align*}
\caption{Coefficients of the new third order \added[id=1]{(parallel)} ADI-GARK method. Here $\gamma \approx 0.43586652150845900$ is the middle root of the polynomial $0 = -1 + 9 \gamma - 18 \gamma^2 + 6 \gamma^3$.}
\label{tab:ADI-GARK-3}
\end{table}

\subsection{\added[id=1]{(Parallel)} ADI-GARK method of order 4}
\added[id=2]{We follow a similar method derivation process for a fourth order ADI-GARK method.  The implicit part is the L-stable method ESDIRK4(3)6L[2]SA from \cite[Table 16]{kennedy2016diagonally}.  This must be paired with a six stage explicit Runge--Kutta method of order four.  Again, we enforce the first column simplifying assumption $D(1)$ but now need the coupling conditions \cref{eqn:GARK-coupling-4}.  This leaves four unspecified coefficients that we use to control the stability.  We minimize the value of $|R( [z,z] )|$ at a sample of points along the imaginary axis so that it is stable in the entire left-half plane.  \Cref{tab:ADI-GARK-4} gives the resulting coefficients, and \cref{fig:stab-Order4} plots its stability.}
\begin{landscape}
\begin{table}[h]
\centering
	\begin{align*}
		\A\I &= \begin{bmatrix}
			0 & 0 & 0 & 0 & 0 & 0 \\
			\frac{1}{4} & \frac{1}{4} & 0 & 0 & 0 & 0 \\
			\frac{1-\sqrt{2}}{8} & \frac{1-\sqrt{2}}{8} & \frac{1}{4} & 0 & 0 & 0 \\
			\frac{5-7 \sqrt{2}}{64} & \frac{5-7 \sqrt{2}}{64} & \frac{7 \left(\sqrt{2}+1\right)}{32} & \frac{1}{4} & 0 & 0 \\
			\frac{-54539 \sqrt{2}-13796}{125000} & \frac{-54539 \sqrt{2}-13796}{125000} & \frac{132109 \sqrt{2}+506605}{437500} & \frac{166 \left(376 \sqrt{2}-97\right)}{109375} & \frac{1}{4} & 0 \\
			\frac{1181-987 \sqrt{2}}{13782} & \frac{1181-987 \sqrt{2}}{13782} & \frac{47 \left(1783 \sqrt{2}-267\right)}{273343} & \frac{16 \left(-3525 \sqrt{2}+22922\right)}{571953} & \frac{15625 \left(-376 \sqrt{2}-97\right)}{90749876} & \frac{1}{4} \\
		\end{bmatrix}, \\
		\A\E &= \resizebox{1.48\textwidth}{!}{$\begin{bmatrix}
			0 & 0 & 0 & 0 & 0 & 0 \\
			\frac{1}{2} & 0 & 0 & 0 & 0 & 0 \\
			\frac{4}{7}-\frac{1}{2 \sqrt{2}} & -\frac{1}{14} & 0 & 0 & 0 & 0 \\
			\frac{192440351 \sqrt{2}+245255777}{1090446224} & \frac{1059385241-192440351 \sqrt{2}}{1090446224} & -\frac{4}{7} & 0 & 0 & 0 \\
			\frac{3246103358815879 \sqrt{2}-4074461458752694}{1911688536450000} & \frac{15031561460125012-11088311262828073 \sqrt{2}}{1911688536450000} & \frac{1307034650668699 \sqrt{2}-1700986476469053}{318614756075000} & \frac{11}{17} & 0 & 0 \\
			\frac{2357123976102849118-3355327406349634955 \sqrt{2}}{1982691401525245488} & \frac{4815717108798877157 \sqrt{2}-9817340273693398308}{1982691401525245488} & \frac{3722435241465127195-759937254896120301 \sqrt{2}}{991345700762622744} & \frac{7576400 \sqrt{2}+387641523}{385686973} & \frac{625 \left(376 \sqrt{2}+97\right)}{22687469} & 0 \\
		\end{bmatrix}$}, \\
		\b\I &= \begin{bmatrix}
			\frac{1181-987 \sqrt{2}}{13782} & \frac{1181-987 \sqrt{2}}{13782} & \frac{47 \left(1783 \sqrt{2}-267\right)}{273343} & \frac{16 \left(-3525 \sqrt{2}+22922\right)}{571953} & \frac{15625 \left(-376 \sqrt{2}-97\right)}{90749876} & \frac{1}{4}
		\end{bmatrix}, \\
		\c\I &= \begin{bmatrix}
			0 & \frac{1}{2} & \frac{2-\sqrt{2}}{4} & \frac{5}{8} & \frac{26}{25} & 1
		\end{bmatrix}.
	\end{align*}
\caption{Coefficients of the new fourth order \added[id=1]{(parallel)} ADI-GARK method.}
\label{tab:ADI-GARK-4}
\end{table}
\end{landscape}

\begin{figure}
    \centering
    \begin{subfigure}[b]{0.45\textwidth}
        \centering
            \includegraphics[width=4.5cm]{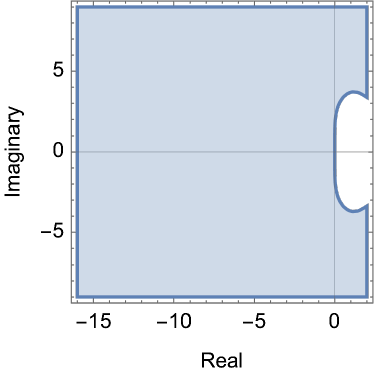}
            \caption{Order 3 ADI-GARK method from \cref{tab:ADI-GARK-3}}
            \label{fig:stab-Order3}
    \end{subfigure}
	\hfil
    \begin{subfigure}[b]{0.45\textwidth}
        \centering
            \includegraphics[width=4.75cm]{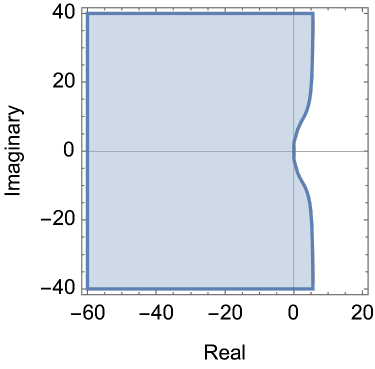}
            \caption{Order 4 ADI-GARK method from \cref{tab:ADI-GARK-4}}
            \label{fig:stab-Order4}
\end{subfigure}
\caption{\added[id=2]{The linear stability regions for new ADI-GARK methods applied to a two-way partitioned system with partitions of equal stiffness.  That is, the $z \in \Cplx$ such that $|R( [z, z] )| \leq 1$.}}
\label{fig:stab-plots}
\end{figure}

	\section{Numerical experiments}
\label{sec:numerical_experiments}
%%%%%%%%%%%%%%%%%%%%%%%%%%%%%%%%%%%%%%%%%%%%%%%%%%%%%%%%%%%%%%%
We test the accuracy of the newly developed ADI-GARK methods on two parabolic PDEs from \cite[Section 7]{Sandu_2019_GLM-ADI}.  The first is the two-dimensional problem
\begin{equation} \label{eqn:2d_heat_eq}
	\begin{split}
		u_t &= u_{xx} + u_{yy} + h(x,y,t) \\
		h(x,y,t) &= e^t (1-x) x (1-y) y+e^t \left(\left(x+\tfrac{1}{3}\right)^2+\left(y+\tfrac{1}{4}\right)^2-4\right) \\
		&\quad + 2 e^t  (1-x) x+2 e^t (1-y) y,
	\end{split}
\end{equation}
posed on the unit square $[0, 1] \times [0, 1]$.  Boundary and initial conditions come from the exact solution
\begin{equation*}
	u(x,y,t) = e^t (1-x) x (1-y) y+e^t \left(\left(x+\tfrac{1}{3}\right)^2+\left(y+\tfrac{1}{4}\right)^2\right).
\end{equation*}
The second problem is three-dimensional and reads
\begin{equation} \label{eqn:3d_heat_eq}
	\begin{split}
	    u_t &= u_{xx} + u_{yy}+ u_{zz} + g(x,y,z,t), \\
	    g(x,y,z,t) &= e^t (1-x) x (1-y) y (1-z) z +2 e^t (1-x) x (1-y) y \\
	    &\quad +2 e^t (1-x) x (1-z) z +2 e^t (1-y) y (1-z)z-6 e^t\\
	    & \quad +e^t \left(\left(x+\tfrac{1}{3}\right)^2+\left(y+\tfrac{1}{4}\right)^2+\left(z+\tfrac{1}{2
	    }\right)^2\right),
    \end{split}
\end{equation}
on the unit cube $[0, 1] \times [0, 1] \times [0, 1]$.  For \cref{eqn:3d_heat_eq} the exact solution is
\begin{equation*}
	\begin{split}
		u(x,y,z,t) &=e^t (1-x) x (1-y) y (1-z) z \\
		& \quad +  e^t \mleft(\mleft(x+\tfrac{1}{3}\mright)^2+\mleft(y+\tfrac{1}{4}\mright)^2+\mleft(z+\tfrac{1}{2}\mright)^2\mright).
	\end{split}
\end{equation*}
For both \cref{eqn:2d_heat_eq,eqn:3d_heat_eq}, we discretize spatial derivatives with second order central finite differences on a uniform mesh with $N_p$ points in each direction.  Given the exact solutions are quadratic in space, this ensures there are no spatial errors. \added[id=1]{The 2D problem is partitioned such that $f^{\{1\}}$  is the discretization of the directional derivative along the $x$-axis and $f^{\{2\}}$ discretizes $u_{yy} +h(x,y,t)$. Similarly, in the 3D problem, $f^{\{1\}}$ and  $f^{\{2\}}$ are finite difference operators for $u_{xx}$ and $u_{yy}$, respectively, while $f^{\{3\}} $ represents the remaining terms $u_{zz} + g(x,y,z,t)$.} The integration timespan is $[0, 1]$, and temporal error is measured in the $\ell^2$ norm with respect to the exact solution at \added[id=1]{$t=1$ evaluated at} the mesh points.

\Cref{fig:ADI_2D_conv,fig:ADI_3D_conv} show convergence plots for the  methods documented in \cref{sec:new_methods} applied to the 2D and 3D problems respectively when the ADI structure in \cref{definition:ADI-GARk} is considered. The uniform mesh and the singly-diagonally implicit structure of the method  allows us to use the same matrix factorization when computing different directional stages. For small values of $N_p$  the methods work at their nominal order of convergence. However, numerical convergence rates in \cref{fig:order3-conv-2D,fig:order4-conv-2D,fig:order4-conv-3D} indicate that as the mesh size gets smaller, the problem becomes stiffer and numerical order reduction is observed. \added[id=2]{An order reduction in the classical convergence order of one-step methods when applied to PDE problems with time dependent boundary conditions is generally present. This can be seen in the papers by Ostermann and Roche \cite{Roche92} and by Lubich and Ostermann \cite{Lubich95,Lubich1995-Convergence}, where it is shown that for Runge--Kutta methods and Rosenbrock-type methods an order reduction takes place even in the case of time-independent boundary conditions. The order reduction is more dramatic for the case of time-dependent BCs. Besides in case of splitting methods (e.g. the GARK methods here considered) the effects of the order reduction are  quite often more pronounced due to the effect of the additional errors introduced in the splitting terms.  See, e.g. \cite{Gonzalez_2020_AMF}, where some convergence results  about the PDE-order (in several $L_p$ norms) of some splitting W-methods are given.}

\Cref{fig:P_ADI_2D_conv} shows convergence results for methods of orders 3 and 4 when the parallel ADI scheme described in \cref{definition:P-ADI-GARk} is used on the 2D problem \cref{eqn:2d_heat_eq}. This scheme has the added computational benefit that directional stages at the same time argument $t_{n}+ c_i h $ can be computed in parallel. For the set of mesh sizes used in this experiment, we observe the classical order of convergence for the  methods.
\begin{figure}[tbhp]
	\centering
	\begin{subfigure}[t]{.48\linewidth}
		\begin{tikzpicture}
			\begin{loglogaxis}[height=2in, grid=major, xlabel={Steps}, ylabel={Error}, 
			legend entries={ $N_p=16$\, Order 3.0,  $N_p=32$\, Order 2.9,
				$N_p=64$\, Order 2.6}]
			%\addplot table [x index=0, y index=1] {\datadir/2D/ADI_RK3_opt_2D.txt};
			\addplot table [x index=0, y index=1,col sep=comma] {\datadir/2D/ADI_RK3_opt_2D.txt};
			\addplot table [x index=0, y index=2,col sep=comma] {\datadir/2D/ADI_RK3_opt_2D.txt};
			\addplot table [x index=0, y index=3,col sep=comma] {\datadir/2D/ADI_RK3_opt_2D.txt};
			%\draw[dashed] (axis cs:6.0e3,2e-11) -- node[below]{3} (axis cs:2*6.0e3, 1/8*2e-11);
			\end{loglogaxis}
		\end{tikzpicture}
		\caption{Order 3 method }
        \label{fig:order3-conv-2D}
	\end{subfigure} \hfill
	\begin{subfigure}[t]{.48\linewidth}
		\begin{tikzpicture}
			\begin{loglogaxis}[height=2in, grid=major, xlabel={Steps}, ylabel={Error},
			legend entries={$N_p=16\quad$\, Order 4.0,  $N_p=32\quad$\, Order 3.9,
			$N_p=64\quad$\, Order 3.6}]
			\addplot table [x index=0, y index=1,col sep=comma] {\datadir/2D/ADI_RK4_opt_N16.txt};
			\addplot table [x index=0, y index=2,col sep=comma] {\datadir/2D/ADI_RK4_opt_2D.txt};
			\addplot table [x index=0, y index=3,col sep=comma] {\datadir/2D/ADI_RK4_opt_2D.txt};
			%\draw[dashed] (axis cs:6.3e3,1e-13) -- node[below]{4} (axis cs:2*6.3e3, 1/16*1e-13);
			\end{loglogaxis}
		\end{tikzpicture}
		\caption{Order 4 method }
        \label{fig:order4-conv-2D}
	\end{subfigure} \hfill
	\caption{Convergence plots for ADI-GARK methods on the 2D test problem \cref{eqn:2d_heat_eq}.}
	\label{fig:ADI_2D_conv}
\end{figure}
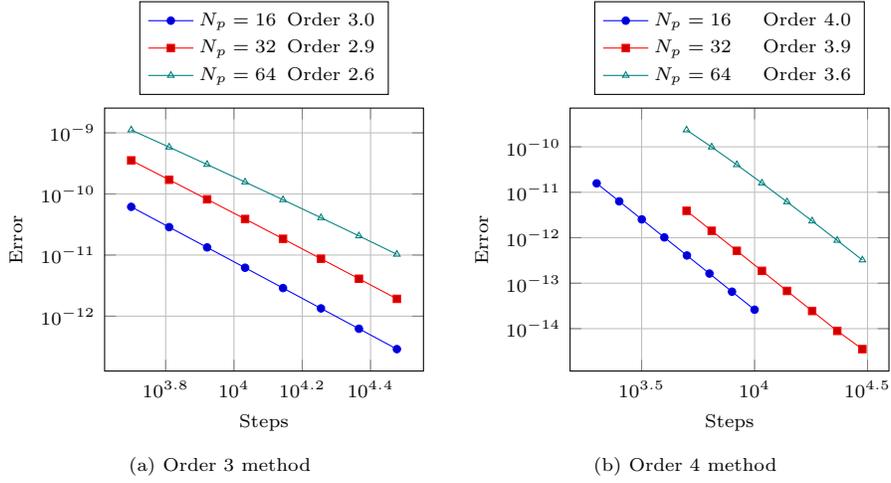
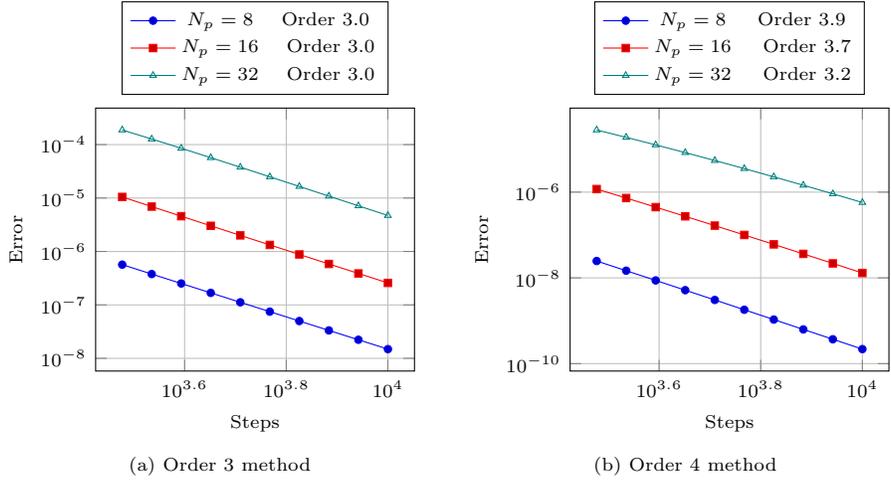
\begin{figure}[tbhp]
    \centering
    \begin{subfigure}[t]{.48\linewidth}
        \begin{tikzpicture}
        \begin{loglogaxis}[height=2in, grid=major, xlabel={Steps}, ylabel={Error},
        legend entries={$N_p=8\quad$ Order 3.0,  $N_p=16\quad$  Order 3.0,  $N_p=32\quad$  Order 3.0,
            $N_p=64\quad$, Reference slope 3}]
        \addplot table [x index=0, y index=1] {\datadir/3D/ADI_RK3_3D.txt};
        \addplot table [x index=0, y index=2] {\datadir/3D/ADI_RK3_3D.txt};
        \addplot table [x index=0, y index=3] {\datadir/3D/ADI_RK3_3D.txt};
        \draw[dashed] (axis cs:5.0e3,5e-3) -- node[below]{3} (axis cs:2*5.0e3, 5/8*1e-3);
        \end{loglogaxis}
        \end{tikzpicture}
        \caption{Order 3 method }
        \label{fig:order3-conv-3D}
    \end{subfigure} \hfill
    \begin{subfigure}[t]{.48\linewidth}
        \begin{tikzpicture}
        \begin{loglogaxis}[height=2in, grid=major, xlabel={Steps}, ylabel={Error},
        legend entries={$N_p=8\quad$\, Order 3.9,  $N_p=16\quad$\, Order 3.7,  $N_p=32\quad$\, Order 3.2,
            $N_p=64\quad$}]
        \addplot table [x index=0, y index=1] {\datadir/3D/ADI_RK4_3D.txt};
        \addplot table [x index=0, y index=2] {\datadir/3D/ADI_RK4_3D.txt};
        \addplot table [x index=0, y index=3] {\datadir/3D/ADI_RK4_3D.txt};
        %\draw[dashed] (axis cs:5.0e3,1e-4) -- node[below]{4} (axis cs:2*5.0e3, 1/16*1e-4);
        \end{loglogaxis}
        \end{tikzpicture}
        \caption{Order 4 method }
        \label{fig:order4-conv-3D}
    \end{subfigure} \hfill
    \caption{Convergence plots for ADI-GARK methods on the 3D test problem \cref{eqn:3d_heat_eq}.}
    \label{fig:ADI_3D_conv}
\end{figure}
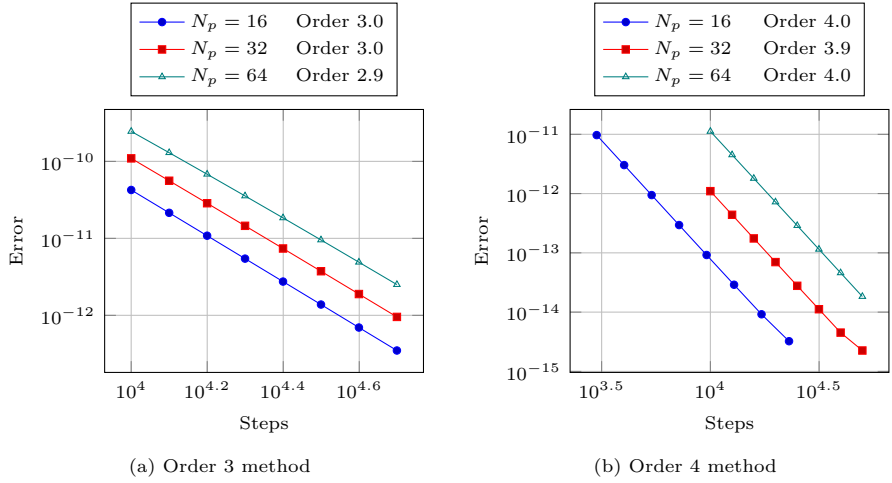
\begin{figure}[tbhp]
    \centering
    \begin{subfigure}[t]{.48\linewidth}
        \begin{tikzpicture}
        \begin{loglogaxis}[height=2in, grid=major, xlabel={Steps}, ylabel={Error},
        legend entries={  $N_p=16\quad$ Order 3.0,  $N_p=32\quad$ Order 3.0,
            $N_p=64\quad$ Order 2.9, Reference slope 3}]
        %\addplot table [x index=0, y index=1] {\datadir/Parallel/P_RK3_opt_2D.txt};
        \addplot table [x index=0, y index=2] {\datadir/Parallel/P_RK3_opt_2D.txt};
        \addplot table [x index=0, y index=3] {\datadir/Parallel/P_RK3_opt_2D.txt};
        \addplot table [x index=0, y index=4] {\datadir/Parallel/P_RK3_opt_2D.txt};
        %\draw[dashed] (axis cs:12e3,10e-12) -- node[below]{3} (axis cs:2*12e3, 10/8*1e-12);
        \end{loglogaxis}
        \end{tikzpicture}
        \caption{Order 3 method }
        \label{fig:order3-conv-P}
    \end{subfigure} \hfill
    \begin{subfigure}[t]{.48\linewidth}
        \begin{tikzpicture}
        \begin{loglogaxis}[height=2in, grid=major, xlabel={Steps}, ylabel={Error},
        legend entries={  $N_p=16\quad$\, Order 4.0,  $N_p=32\quad$\, Order 3.9,
            $N_p=64\quad$\, Order 4.0}]
        %\addplot table [x index=0, y index=1] {\datadir/Parallel/P_RK4_opt_2D.txt};
        \addplot table [x index=0, y index=1] {\datadir/Parallel/RK4_N16_2D.txt};
        \addplot table [x index=0, y index=3] {\datadir/Parallel/P_RK4_opt_2D.txt};
        \addplot table [x index=0, y index=4] {\datadir/Parallel/P_RK4_opt_2D.txt};
        %\draw[dashed] (axis cs:8e3,1e-13) -- node[below]{4} (axis cs:2*8e3, 1/16*1e-13);
        \end{loglogaxis}
        \end{tikzpicture}
        \caption{Order 4 method }
    \end{subfigure} \hfill
    \caption{Convergence plots for parallel ADI-GARK methods on the 2D test problem \cref{eqn:2d_heat_eq}.}
    \label{fig:P_ADI_2D_conv}
    \label{fig:order4-conv-P}
\end{figure}
%
%%%%%%%%%%%%%%% Timing experiment
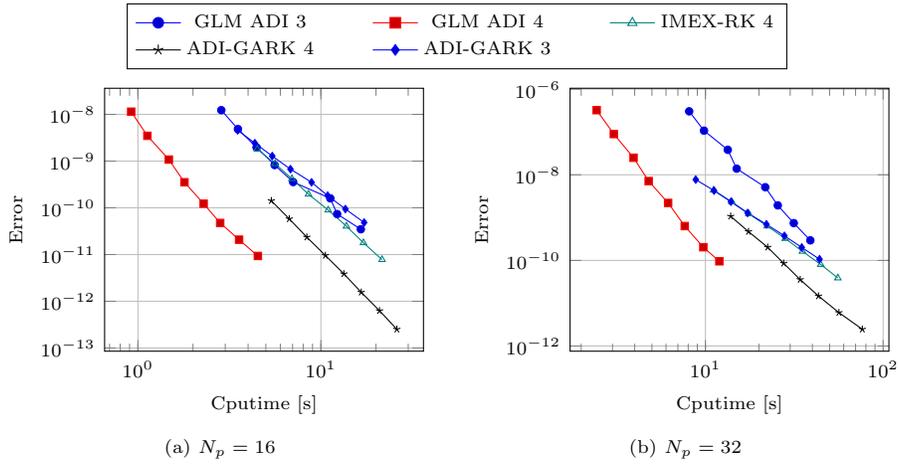
\begin{figure}[tbhp]
	\centering
	\pgfplotslegendfromname{leg:ADI_2D_timing}
	
	\begin{subfigure}[t]{.48\linewidth}
		\begin{tikzpicture}
			\begin{loglogaxis}[height=2in, grid=major, xlabel={Cputime [s]}, ylabel={Error}, 
			legend entries={GLM ADI 3, GLM ADI 4, IMEX-RK 4, ADI-GARK 4, ADI-GARK 3 },
			legend style={legend columns=3,/tikz/every even column/.append style={column sep=2em}},
			legend to name={leg:ADI_2D_timing}]
			%\addplot table [x index=0, y index=1] {\datadir/2D/ADI_RK3_opt_2D.txt};
 			\addplot table [x index=5, y index=2,col sep=tab] {\datadir/Timing/GLM_ADI3.txt};
			\addplot table [x index=5, y index=2,col sep=tab] {\datadir/Timing/GLM_ADI4.txt};
 			\addplot table [x index=5, y index=2,col sep=tab] {\datadir/Timing/ADI_RK4_2D.txt};
 			\addplot table [x index=5, y index=2,col sep=tab] {\datadir/Timing/ADI_GARK4_2D.txt};
 			\addplot table [x index=5, y index=2,col sep=tab] {\datadir/Timing/ADI_GARK3_2D.txt};
% 			%\draw[dashed] (axis cs:6.3e3,1e-13) -- node[below]{4} (axis cs:2*6.3e3, 1/16*1e-13);
			\end{loglogaxis}
		\end{tikzpicture}
		\caption{$N_p=16$}
        \label{fig:timing-order-3}
	\end{subfigure} \hfill
	\begin{subfigure}[t]{.48\linewidth}
		\begin{tikzpicture}
			\begin{loglogaxis}[height=2in, grid=major, xlabel={Cputime [s]}, ylabel={Error}]
			%\addplot table [x index=0, y index=1] {\datadir/2D/ADI_RK3_opt_2D.txt};
 			\addplot table [x index=6, y index=3,col sep=tab] {\datadir/Timing/GLM_ADI3.txt};
			\addplot table [x index=6, y index=3,col sep=tab] {\datadir/Timing/GLM_ADI4.txt};
 			\addplot table [x index=6, y index=3,col sep=tab] {\datadir/Timing/ADI_RK4_2D.txt};
 			\addplot table [x index=6, y index=3,col sep=tab] {\datadir/Timing/ADI_GARK4_2D.txt};
 			\addplot table [x index=6, y index=3,col sep=tab] {\datadir/Timing/ADI_GARK3_2D.txt};
			%\draw[dashed] (axis cs:6.0e3,2e-11) -- node[below]{3} (axis cs:2*6.0e3, 1/8*2e-11);
			\end{loglogaxis}
		\end{tikzpicture}
		\caption{$N_p=32$}
        \label{fig:timing-order-4}
	\end{subfigure} \hfill
	\caption{ Error versus  cputime for ADI-GARK methods on the 2D test problem \cref{eqn:2d_heat_eq} compared to GLM ADI methods from \cite{Sandu_2019_GLM-ADI} and an ADI method created from the IMEX pair reported in  \cite[Example 3]{Sandu_2015_GARK}.} 
	\label{fig:ADI_2D_timing}
\end{figure}
\added[id=2]{
We perform efficiency experiments using the new methods in \cref{sec:new_methods}. We compare the runtime and the error of the final solution against GLM-ADI methods reported in \cite{Sandu_2019_GLM-ADI}. In order to highlight the benefits of design strategies utilized for creating the new methods, we also compare them with a generic fourth order IMEX method from  \cite[Example 3]{Sandu_2015_GARK} used in ADI mode.  \Cref{fig:ADI_2D_timing} shows the  error versus cputime for $N_p \in \{ 16,32 \}$ on the 2D test problem \cref{eqn:2d_heat_eq}. We observe good performance for the ADI-GARK  methods, only being outperformed by the GLM-ADI 4. We also note that, owing to its optimized error, the ADI-GARK 3 method shows performance close to that of the IMEX-RK 4  method.}
	%%%%%%%%%%%%%%%%%%%%%%%%
\section{Conclusions}
\label{sec:conclusions}
%%%%%%%%%%%%%%%%%%%%%%%%

This work introduces the implicit-implicit GARK family of methods in the general-structure additive Runge--Kutta framework. The IMIM-GARK family is of interest since it provides the general computational template for implicit time integration based on splitting.  Existing partitioned approaches such as fractional step, alternating direction implicit integration, operator splitting, and locally one dimensional integration, are formulated as IMIM-GARK methods. All these methods can be studied in a unified way using the order conditions and stability analyses provided herein. New splitting methods of (classical) order three and four with optimized stability and error constants are developed using the IMIM-GARK framework. Numerical experiments verify the accuracy \added[id=2]{and the efficiency} of these new schemes.  
	
	\bibliographystyle{model1b-num-names.bst}
	\bibliography{./Bib/main,./Bib/sandu}
	
	\appendix
	\section{GARK order conditions} \label{app:order_conditions}

General order conditions for partitioned GARK methods \cref{eqn:GARK} are described in \cite[Theorem 2.6]{Sandu_2015_GARK}.  Let
\begin{equation*}
	\A\comp{\sigma,\nu } \one{s\comp{\nu}}  = \c\comp{\sigma,\nu} \quad \forall \sigma,\nu,
\end{equation*}
\added[id=2]{and assume $\c\comp{\sigma} = \c\comp{\sigma,\sigma}$ for $\sigma = 1, \dots \nparts$.  This is weaker than the internal consistency condition \cref{eqn:internal-consistency}.}  Then the order condition up to order four are
\begin{subequations}
    \label{eqn:GARK-order-conditions}
	\begin{align}
		\b\comp{\sigma}*  \one{s\comp{\sigma}}  =  1 & \quad \forall \sigma,  & (\textnormal{order}~ 1) \label{eqn:order1} \\
		\b\comp{\sigma}*  \c\comp{\sigma,\nu}  =  \frac{1}{2}, & \quad \forall \sigma,\nu, & (\textnormal{order}~ 2) \\
		\b\comp{\sigma}*  \left(  \c\comp{\sigma,\nu} \times \c\comp{\sigma,\mu} \right)
		 =  \frac{1}{3}, & \quad \forall \sigma,\nu,\mu,   & (\textnormal{order}~ 3) \\
		\b\comp{\sigma}*  \A\comp{\sigma,\nu}   \c\comp{\nu,\mu} 
		 =  \frac{1}{6}, & \quad \forall \sigma,\nu,\mu,  & (\textnormal{order}~ 3)\\
		\b\comp{\sigma}*    \left( \c\comp{\sigma,\lambda} \times \c\comp{\sigma,\mu} \times \c\comp{\sigma,\nu} \right)
		=  \frac{1}{4}, & \quad \forall \lambda,\sigma,\nu,\mu,   & (\textnormal{order}~ 4)\\
		\left( \b\comp{\sigma} \times \c\comp{\sigma,\mu}  \right)^\tp  \A\comp{\sigma,\nu}     \c\comp{\nu,\lambda} 
		=  \frac{1}{8}, & \quad \forall \lambda,\sigma,\nu,\mu,  & (\textnormal{order}~ 4)\\
		\b\comp{\sigma}*  \A\comp{\sigma,\lambda}     \left(  \c\comp{\lambda,\mu} \times \c\comp{\lambda,\nu} \right)
		=  \frac{1}{12}, & \quad \forall \lambda,\sigma,\nu,\mu,  & (\textnormal{order}~ 4)\\
		\b\comp{\sigma}*  \A\comp{\sigma,\lambda}   \A\comp{\lambda,\nu}   \c\comp{\nu,\mu} 
		=  \frac{1}{24}, &\quad \forall \lambda,\sigma,\nu,\mu,  & (\textnormal{order}~ 4) 
	\end{align}
\end{subequations}
where $\lambda,\sigma,\nu,\mu \in \{1,2,\ldots,\nparts \}$.
	
	\ifreport
	%%%%%%%%%%%%%%%%%%%%%
\section{Preserving Equilibria}
%%%%%%%%%%%%%%%%%%%%%

\begin{theorem}
	Consider an autonomous system with
	\begin{equation} \label{eqn:steady-state-f}
		\sum_{m=1}^{\nparts} f\comp{m}(y^*) =  0
	\end{equation}
	such that $y(t) = y^*$ is a constant solution to \cref{eqn:additive-ode}.  If a GARK method is internally consistent and first order accurate \cref{eqn:order1}, it will produce a numerical solution that stays at the steady state $y^*$.
\end{theorem}

\begin{proof}
	Consider the nonlinear system from simultaneously solving all of the GARK stage equations \cref{eqn:GARK-stage}.  This admits the solution $Y\comp{q}_i = y^*$ for all $q =1,\dots,\nparts$ and $i=1,\dots,s\comp{q}$ since
	\begin{equation*}
		y^* + \Dt \sum_{m=1}^\nparts \sum_{j=1}^{s\comp{m}} a\comp{q,m}_{i,j} \, f\comp{m}(y^*)
		= y^* + \Dt \, c\comp{q}_i \sum_{m=1}^\nparts f\comp{m}(y^*)
		= y^*.
	\end{equation*}
	Now using \cref{eqn:order1}, we have that
	\begin{equation*}
		y_{n+1}
		= y^* + \Dt \sum_{q=1}^\nparts \sum_{i=1}^{s\comp{m}} b\comp{q}_{i} \, f\comp{q}(y^*)
		= y^* + \Dt \sum_{q=1}^\nparts f\comp{q}(y^*)
		= y^*.
	\end{equation*}
	\qed
\end{proof}

	\fi

\end{document}